\documentclass{article}

\usepackage[utf8]{inputenc}

\usepackage{float}
\usepackage{graphicx}
\usepackage{caption}
\usepackage{subcaption}
\usepackage{hyperref}
\usepackage{amssymb}
\usepackage{amsmath}
\usepackage{amsthm}
\usepackage{comment}

\numberwithin{equation}{section}
\def\R{\mathbb{R}}

\def\N{\mathbb{N}}

\usepackage{mfirstuc}
\usepackage{bbold}
\usepackage{dsfont}
\usepackage{amsfonts}
\usepackage{amsmath}
\usepackage{multirow}
\usepackage{multicol}
\usepackage{longtable}
\usepackage{array}
\usepackage{color}
\usepackage{colortbl}
\usepackage{fullpage}
\usepackage{graphicx}

\newtheorem{definition}{Definition}[section]
\newtheorem{theorem}{Theorem}[section]

\newtheorem{Proposition}{Proposition}[section]

\newtheorem{lemma}[theorem]{Lemma}
\newtheorem{Remark}[theorem]{Remark}

\usepackage[UKenglish]{babel}

\usepackage{hyperref}
\newcommand{\footremember}[2]{%
    \footnote{#2}
    \newcounter{#1}
    \setcounter{#1}{\value{footnote}}%
}

\title{A dynamic optimal reinsurance strategy with capital injections in the Cram\'er-Lundberg model}

\author{%
Zakaria Aljaberi\footremember{Zakaria}{ENIT, Tunisia. Email: zakaria.aljaberi@ieee.org}%
  \and Asma Khedher\footremember{Asma}{Korteweg-de Vries Institute for Mathematics, University of Amsterdam, Postbus 94248, NL–1090 GE Amsterdam, The Netherlands. Email: A.Khedher@uva.nl}%
  \and Mohamed Mnif\footremember{Mohamed}{ENIT, Tunisia. Email: mohamed.mnif@enit.utm.tn}%
  }

\date{\today}

\begin{document}

\maketitle

\begin{center}
    \textbf{Abstract}
\end{center}

\noindent In this article we consider the surplus process of an insurance company within the Cramér–Lundberg framework. We study the optimal reinsurance strategy and dividend distribution of an insurance company under proportional reinsurance, in which capital injections are allowed. Our aim is to find a general dynamic reinsurance strategy that maximises the expected discounted cumulative dividends until the time of passage below a given level, called ruin. These policies consist in stopping at the first time when the size of the overshoot below 0 exceeds a certain limit $a$, and only pay dividends when the reserve reaches an upper barrier $b$. Using analytical methods, we identify the value function as a particular solution to the associated Hamilton Jacobi Bellman equation. This approach leads to an exhaustive and explicit characterisation of optimal policy. The proportional reinsurance is given via comprehensive structure equations. Furthermore we give some examples illustrating the applicability of this method for proportional reinsurance treaties.

\section{Introduction}
In this paper, we study a dynamic optimal control problem in a setting of a generalised Cram\'er-Lundberg reinsurance model in which capital injections incur proportional cost. The classical Cram\'er-Lundberg  model is a Markov process driven by a drift term describing the premium rate and a compound Poisson process describing the arrival of claims (see e.g., \cite{mandjescramer} for more about this model). The objective is to maximise a discounted dividend payment minus the penalised discounted capital injection until the time of the ruin. The work in the current paper generalises the model in \cite{avram2021equity} to the case of {\it dynamic} reinsurance, where we consider a stochastic control. This stochastic control enters in the dynamics of our model.\\
The problem of the minimisation of the ruin probability by means of reinsurance with proportional cost is well studied in the literature see e.g., \cite{schmidli2001optimal, schmidli2007stochastic, hipp2003optimal, azcue2014stochastic, mnif2005optimal, azcue2005optimal, jgaard1999controlling} where both diffusion and classical Cram\'er-Lundberg models are considered. 
Some additional results with a focus on non-proportional reinsurance contracts and stochastic differential games are given in \cite{hipp2010optimal}. Another modelling approach was to approximate the Cram\'er-Lundberg process by a simple process with exponential jumps called in the literature Vylder approximation, see, e.g., \cite{cohen2020rate}.\\
For risk models without bankruptcy, where shareholders inject capital to cover the deficit, we refer e.g.~to \cite{li2015optimal, avram2022optimizing, avanzi2011optimal, kulenko2008optimal, eisenberg2011optimal}.
Using a different criterion to assess the performance of an insurance portfolio,
\cite{eisenberg2010optimal} thoroughly covers a variety of capital injection minimisation problems
under both the classical Cram\'er-Lundberg risk model and its diffusion approximation where
the insurer has the possibility to dynamically reinsure its risk. 
Combining dividend pay-outs maximisation with proportional risk exposure
reduction, \cite{schal1998piecewise} formulated a piece-wise deterministic Markov model where
only jumps but not the deterministic flow can be controlled. The aforementioned
references deal with optimal reinsurance for continuous-time risk
processes. For a discrete-time insurance model, see e.g., \cite{schal2004discrete} where reinsurance and investments in a financial market with the intention to either maximise the
expected exponential utility or minimise the ruin probability are investigated. 
Optimal control problems considered by maximising expected utility of terminal reserve process are studied, e.g., in \cite{irgens2004optimal,touzi2000optimal, di2024utility}.\\
 In this paper, we investigate the problem of optimal risk control under capital injections, where the objective is to \emph{maximize dividend payments}. Our main contribution, comparable to the work in \cite{avram2021equity}, is the study of this control problem within a \emph{generalized classical risk model}, in which the drift and jump sizes of the surplus process depend on a stochastic process $(\alpha_t)_{t \ge 0}$ representing the reinsurance scheme (see equation~\eqref{Xalpha} in Section~\ref{sec:problem} for a detailed description of the model). To solve this optimization problem, we employ two approaches. In the first, we use the powerful framework of \emph{viscosity solutions}, which allows us to characterize the value function as the unique solution of the associated system of Hamilton--Jacobi--Bellman (HJB) equations within the class of uniformly continuous functions with linear growth. However, the structure of the optimal control cannot be derived directly via the \emph{dynamic programming principle}. To address this issue, we impose additional assumptions on the distribution of the L\'evy process in the model and restrict the optimal reinsurance strategy to change at discrete time points. This enables us to exploit the characteristic function of the L\'evy process, leading to semi-explicit expressions for the optimal control $(\alpha_{t_i}^*)_{1 \le i \le N}$.\\
In our study, we consider a dynamic control problem with the aim of determining, at each time instant $t \geq 0$, a control parameter $\alpha_t$ which specifies a reinsurance scheme chosen from an available set of schemes. The choice of this scheme simultaneously determines the extent to which the insurer wishes to reduce its risk exposure and the additional cost incurred for this protection, in the form of a reinsurance premium. As a by-product, the optimal strategies consist of injecting capital when the size of the deficit below $0$ does not exceed a certain limit $a$, and paying dividends only when the reserve reaches a barrier $b$. The first results addressing this type of problem are due to \cite{de1957impostazione}, where the maximization of the expected value of cumulative discounted dividends up to the time of ruin (i.e., the passage below a given level) is studied. This can be seen as a special case of the above framework, in which the maximum targeted ruin severity is $a = 0$ (or equivalently, the cost of borrowing money is infinite). In \cite{shreve1984optimal}, this problem is treated in the case $a = \infty$.  
For $a \notin \{0, \infty\}$, the authors in \cite{avram2022optimizing} were the first to quantify the optimal buffer (severity of ruin) when bankruptcy should replace individual bailouts, while studying optimality. For a different approach that quantifies the optimal policy under a constraint on the expected total bailouts, see \cite{junca2019optimal}. \\
Different approaches have been used in the literature to solve optimal control problems, relying on duality characterizations, the theory of backward stochastic differential equations, and the theory of partial differential equations, where one derives the so-called Hamilton--Jacobi--Bellman (HJB) equations for the value function (see, e.g., \cite{pham2009continuous, azcue2005optimal} for a comprehensive overview of optimal control problems in a Markovian setting in finance and insurance). Another approach is to exploit model properties, such as strong Markovianity and the form of the characteristic function, to conjecture an explicit solution to the optimal control problem \cite{avram2015gerber, avram2021equity}. \\
In this paper, we consider two approaches to tackle our problem. In the first approach, we derive the HJB equation associated with our model and then prove that the value function is both a sub- and super-viscosity solution to this equation. To prove that it is a super-viscosity solution, we use classical results such as the dynamic programming principle (DPP) (see \cite[Lemma~1.2]{azcue2014stochastic}) and the properties of the Cram\'er--Lundberg model, adapting the proof in \cite[Proposition~3]{avram2021equity} to our setting. To show that the value function is a sub-viscosity solution, we rely on the Markovianity of our model and the It\^o formula. Furthermore, we investigate a comparison principle to establish the uniqueness of the viscosity solution to our HJB equation on $(0, \infty)$. It is worth noting that the authors in \cite{avram2021equity} show that the value function in their setting is the lowest absolutely continuous super-solution of their associated HJB equation, but do not study viscosity solutions.\\
In our second approach, we derive explicit expressions for the optimal strategy by restricting the model to exponential jump sizes. In our derivations, we exploit the fact that, under a reinsurance policy with exponential jump sizes, the model is of L\'evy type, and hence its characteristic function is explicitly known. We follow similar steps to those in the proof of \cite[Propositions~7 and Lemma~10]{avram2021equity}, where the idea is to conjecture a family of policies that yields the optimum for all possible parameter values and to compute its expected net present value in terms of certain scale functions. Our derivations become more involved than those in \cite{avram2021equity} due to the stochastic nature of the strategy.  Recently, a related study by \cite{renaud2023optimizationdichotomycapitalinjections} addressed a similar class of stochastic optimization problems by combining the \emph{viscosity solution approach} with \emph{fluctuation identities}. In that work, these two approaches were successfully unified to fully solve the optimization problem. However, in our model, such a unification is not feasible because of the presence of the parameter $\alpha$, which represents the reinsurance strategy. This parameter makes the derivation of an explicit optimal policy extremely cumbersome.  \\
We also conduct a numerical analysis, where we approximate the HJB equation using a finite-difference scheme. The reinsurance strategy is then computed via the Howard algorithm~\cite{howard1960dynamic}, and several examples are presented to illustrate the applicability of this method for proportional reinsurance treaties. The Howard algorithm proves to be very efficient and converges rapidly to the solution. \\
The paper is organized as follows. In Section~\ref{sec:problem}, we introduce the model and formulate the optimal control problem. Section~\ref{sec:HJB} derives the HJB equation associated with the problem and establishes the existence of a unique viscosity solution. In Section~\ref{sec:Guess}, we derive explicit expressions for the value function and discuss the resulting optimal reinsurance strategy. Section~\ref{sec:num} presents several numerical examples in which we compute the optimal value function and optimal strategy. Finally, Section~\ref{sec:conc} concludes the paper.
\section{Problem formulation}\label{sec:problem}
Let $ (\Omega,\mathcal{F},\mathbb{P})$ be a probability space equipped with a filtration $\mathbb{F} = (\mathcal{F}_t)_{t\geq 0}$ satisfying {\it the usual conditions} of right-continuity and completeness (see e.g.~\cite[Chapter I]{protter2004stochastic}). For $t>0$, we denote $\mathcal{F}_{t-}= \sigma\{\mathcal{F}_{t-h},\,\, h>0\}$. Let $\mathbb{N}= \{1,2, \ldots\}$. For $k\in \mathbb{N}$, we denote by $C^k(\R) =
\{f\colon \R \rightarrow \R \ | \ f \mbox{ is } k\mbox{-times continuously differentiable}\}$. Moreover, we denote by
$\mathbb{L}^1(\Omega, \mathbb{P}) = \{f\colon \Omega \rightarrow \R_+ \ | \ f \mbox{ is } \mathcal{F}-\mbox{measurable and } \int_\Omega |f(\omega)|\mathrm{d} \mathbb{P}(\omega) <\infty$\}.   
 We assume that the claims are generated by
a compound Poisson process. More precisely, we consider an integer-valued random measure
$\mu (dt, dz)$ with compensator $\nu(dz) dt$. We assume that $\nu(dz) =\lambda F(dz)$ where $F(dz)$ is a
probability distribution on the set $E=\R_+$ and $\lambda$ is a positive constant. In this
case, the integral, with respect to the random measure $\mu(dt, dz)$, is simply a compound Poisson process: we have 
$\int_0^t\int_E z \mu(dt, dz)= \sum_{i=1}^{N_{t}}U_{i}$ where 
$N=\{N_t,\, t \geq 0\}$ is a Poisson
process with intensity $\lambda$ and $\{U_i, \,i \in \N\}$ is a sequence of random variables with common
distribution $F$ which represent the claim sizes  which are assumed to be independent and identically
distributed and independent of $(N_t)_{t\geq0}$. For every $i \in \mathbb{N}$, the random variable $U_{i}$, has a distribution function $F$ and a finite mean $\mathbb{E}(U_i)$. Later on, we use $U$ as a representative random variable from this distribution.
Consider an insurer whose surplus process is described, in the {\it Cram\'er-Lundberg model}, also known as compound Poisson model or classical risk model, by
\begin{equation}\label{eq-model-X}
    X_{t}= x + ct -\sum_{i=1}^{N_{t}}U_{i}, \qquad t\geq 0\,,
\end{equation}
where $X_{0}= x\geq 0$ is the initial surplus, $c > 0$ is the premium rate.
We  assume that the premium rate $c$ is calculated according to {\it the expected value principle}, i.e., $c = (1+\eta_{1})\lambda \mathbb{E}(U)$, where $\eta_{1} > 0$, is the relative safety loading of the insurer. 
It is well known, that for avoiding almost sure ruin, it is necessary to choose a premium intensity fulfilling the net-profit condition $c > \lambda \mathbb{E}(U)$. Therefore, based on the expected value premium principle we set $c = (1+\eta_{1})\lambda \mathbb{E}(U)$ with a safety loading $\eta_{1} > 0$. For further details on classical problems in risk theory and related topics we refer, e.g., to \cite{asmussen2010ruin, schmidli2007stochastic}.\\
The insurer(cedent), can control the risk by means of some business activities, such as reinsurance. We assume that the insurer has the possibility to take reinsurance in a dynamic way, i.e., at each
time $t\geq 0$, the insurer transfers a portion of the premium income to a reinsurer, who in turn commits to cover a part of the occurred claims. The dynamic reinsurance setup we consider in this paper follows the presentation from \cite{schmidli2007stochastic} and \cite{cani2017optimal}, with some modifications, as required by our problem.\\
 For a control parameter $\alpha \in [0,1]$, and for a claim of size $u$, the amount $u \alpha$ is paid by the insurer and $u-u \alpha$ is paid by the reinsurer. Naturally, when employing reinsurance there are premiums to be paid. We assume that the reinsurer uses for a fixed $\alpha\in [0,1]$, the premium is based on  the expected value principle
\begin{equation}\label{eq:expected-value}
(1+\eta_{2})\lambda\mathbb{E}[U-U\alpha]\,,
\end{equation}
where $\eta_{2}>\eta_{1}$ denotes the safety loading of the reinsurer and $\mathbb{E}$ denotes the expectation with respect to the probability measure $\mathbb{P}$.\\
 From an aggregated risk perspective, if the insurer chooses a reinsurance $\alpha\in [0,1]$ at time $t$, the premium at rate $(1+\eta_{2})\lambda \mathbb{E}[U-U\alpha]$
is paid to the reinsurer. Consequently the premium income of the insurer reduces to $c(\alpha) = c - (1+\eta_{2})\lambda \mathbb{E}[U-U\alpha]$. It follows that $c(\alpha)\leq c$. Notice that full reinsurance leads to a negative premium income, i.e., $c < (1+\eta_{2})\lambda \mathbb{E}[U]$, for $U \in \mathbb{L}^{1}(\Omega, \mathbb{P})$.\\
In this paper, as we work on the {\it proportional reinsurance}, so 
the set of control parameters is $\mathcal{U}=\{\alpha\in [0,1] \ \ | \ \ c(\alpha)\geq 0\}$ for avoiding a negative premium rate. Since $\mathcal{U} \subset \mathbb{R}$ is a closed and bounded set, it holds that $\mathcal{U}$ is compact.


\begin{Remark}
The idea of a dynamic reinsurance strategy can be explained as follows.
At each time instant $t\geq 0$, the insurer chooses a control parameter $\alpha=\alpha_{t}\in \mathcal{U}$. The choice of $\alpha$ simultaneously determines the extent to which the insurer wants to reduce its risk exposure and the additional cost this protection incurs, taking the form of a reinsurance premium. Namely, if a claim occurs at time $t$, the insurer pays $U\alpha_{t}$ and the reinsurer pays the rest, i.e. $U-U\alpha_{t}$. In exchange of this risk transfer, the insurer pays to the reinsurer a reinsurance premium at a rate $(1+\eta_{2})\lambda \mathbb{E}[U-U\alpha]$.
\end{Remark}
\noindent Let $t\geq 0$ be fixed, and  $\alpha=(\alpha_{s})_{s\geq t}$ be an $\mathbb{F}$-predictable, $\mathcal{U}$-valued stochastic process which we call a reinsurance strategy. The dynamics of the controlled surplus process $X^{t,x,\alpha} = (X_{s}^{t,x,\alpha})_{0\leq t\leq s}$ are described by
\begin{equation}\label{Xalpha}
    X_{s}^{t,x, \alpha}= x + \int_{t}^{s}c(\alpha_{u})\, \mathrm{d}u -\sum_{i=N_t +1}^{N_{s}}\alpha_{\tau_{i}}U_{i}, \qquad s\geq t\,,
\end{equation}
where $x\geq 0$, $(N_s)_{s\geq 0}$ and $\{U_i\}_{i\in \mathbb{N}}$ are as described after equation \eqref{eq-model-X},  $\tau_{i}$ is the time of occurrence of the $i$-th claim, and 
\begin{equation}\label{eq:c-alpha}
c(\alpha_{s})=c - (1-\alpha_{s})(1+\eta_{2})\lambda\mathbb{E}(U) = \lambda \mathbb{E}(U)\left[1+\eta_1-(1-\alpha_s)(1+\eta_2)\right], \quad s\geq t.
\end{equation}
We associate with $X_{t,s}^\alpha$, $0\leq t \leq s$ the jump measure $M$ on $\R_+\times \R_+$ describing the jumps of $X_{t,s}^\alpha$, $0\leq t \leq s$. Its compensator is denoted by $\nu$.\\
Notice that the predictability of $\alpha$ induces that it is progressively measurable (see e.g.~\cite[page 45]{revuz2013continuous}). The latter and the fact that the premium rate $c(\cdot)$ is assumed to be bounded by $c$, imply that the integral $\int_{t}^{s}c(\alpha_{u})\, \mathrm{d}u$, $s\geq t\geq 0$, exists in the Lebesgue sense and is progressively measurable. Moreover, since $\alpha$ is $\mathcal{U}$-valued, it holds 
\begin{equation}\label{cphul}
    \alpha_{s}\geq\underline{\alpha}\,,\quad     \mbox{ for all } (s; \omega) \ \  \mbox{a.e.}\,,
\end{equation}
where $\underline{\alpha}$, the lowest admissible retention, is the unique solution of $c(\underline{\alpha})=0$ and is explicitly given by $\underline{\alpha}=(\eta_{2}-\eta_{1})/(\eta_{2}+1).$ In this context the predictability is crucial. That is, at claim time $\tau_i$, the reinsurance parameter is chosen based on the information up to time $\tau_i-$. The predictability of the reinsurance strategy is a natural assumption in this setting, otherwise the insurer could change the reinsurance parameter to $\underline{\alpha}$ reinsurance at the claim occurrence time. The reinsurer would then pay the majority of claims while all premiums would be collected by the insurer. \\
Let $t_0<t_1< \ldots$ be time points.
In this paper we consider $\alpha$ to be such that
\begin{equation}\label{eq:alpha}
\alpha_t = \alpha_{t_i}, \qquad \mbox{for } t \in [t_i, t_{i+1})
\end{equation}
and $\alpha_{t_i} \in \mathcal{F}_{t_i-}$, for all $i \in \mathbb{N}$, i.e., $\alpha$ is $\mathbb{F}$-predictable.
Notice that one can consider the insurance period $[t_{i}, t_{i+1})$ as one period or divide it into n periods $t_{i} = t_{i_{0}} < t_{i_{1}} <... < t_{i_{n}}=t_{i+1}$.
It is up to the insurance company to consider the reinsurance rate as fixed or can be changed during a period of insurance, that is on $[t_{i}, t_{i+1})$, respectively
$$\alpha_{t}=\alpha_{t_i}\,, \quad \mbox{ or } \quad \alpha_{t}= \alpha_{t_{i_{j}}}\mathbb{1}_{[t_{i_{j}}, t_{i_{j+1}})}.$$
Let $(L_s)_{s\geq t}$ and $(I_s)_{s\geq t}$, be $\mathbb{F}$-predictable processes describing respectively dividends and capital injections. Let $\pi=(L_{s},I_{s})_{s\geq t}$. Given a reserve $x$ at time $t \geq 0$ and a policy $(\alpha, \pi)$, the surplus process is modified by dividends and capital injections as
\begin{equation}\label{XLI}
X_{s}^{t,x,\alpha,\pi}=X_{s}^{t,x,\alpha}-L_{t,s}+I_{t,s}\,, \qquad \ \ \forall s\geq t\,,
\end{equation}
where $L_{t,s} = L_{s}- L_t$ and $I_{t,s}$ is defined similarly. Let $\bigtriangleup X_{s} =X_{s}-X_{s-}$ for a given c\`adl\`ag process $(X_t)_{t\geq0}$. We introduce in the following definition the notion of admissible strategies.   
\begin{definition}\label{def:admissible}
A strategy $(\alpha_u, \pi_u)_{u\geq t}$ is said to be {\it admissible} if
\begin{enumerate}
    \item $(\alpha_{u})_{u\geq t}$ is {$\mathbb{F}$ predictable} and  satisfies (\ref{cphul}) and \eqref{eq:alpha}.
    \item The cumulative dividend strategy $(L_u)_{u\geq t}$ is $\mathbb{F}$-predictable, non-decreasing, and c\`adl\`ag (right-continuous  with left-limits), $ L_{0-} = 0$.
    \item The cumulative capital injection process $(I_u)_{u\geq t}$ is $\mathbb{F}$-predictable, non-decreasing, c\`adl\`ag, $ I_{0-} = 0$ and is such that
    $I_u+x\leq \sum_{i=1}^{u}\alpha_{\tau_i}U_i$, for all $u\geq 0$, where $x$ is the initial reserve.
    \item $\bigtriangleup L_{u}\leq X_u^{t,x,\alpha,\pi}$$\, for  $ $(u, \omega)$, a.e. Notice that this inequality expresses the fact that the insurer is not allowed to pay out dividends at time $s\geq t$ which exceed the level of his reserve at this time.
    \end{enumerate}
\noindent Let $a$ be a positive constant. Given a reserve $x\geq -a$ at time $t\geq 0$, the set of all admissible policies will be denoted by $\prod_t(x)$.
\end{definition}
\noindent The time of ruin $\sigma_{t,a}^{x,\alpha, \pi}$ denotes the first time when the controlled surplus process$X_s^{t,x,\alpha,\pi}$ falls behind the negative value $-a$,
\begin{equation}\label{tau}
    \sigma_{t,a}^{x,\alpha, \pi} = \inf\{s > t, \mbox{ such that } X_s^{t,x,\alpha,\pi} < -a \}, \qquad t\geq 0\,.
\end{equation}
Since $(X_s^{t,x,\alpha, \pi})_{s\geq t}$ is a c\`adl\`ag, $\mathbb{F}$-adapted process and $\mathbb{F}$ is assumed to satisfy the usual conditions, it holds that $\sigma_{t,a}^{x,\alpha, \pi}$ is an $\mathbb{F}$-stopping-time, for all $t\geq 0$ (see e.g.~\cite[Chapter I, Proposition 4.6]{revuz2013continuous}). Given a reserve $x\geq 0$ at time $t \geq 0$, we define the expected cumulative discounted dividends until ruin, for $(\alpha, \pi)\in \prod_t(x)$, as
\begin{equation}\label{eq:function-J}
J(t,x,\alpha,\pi)=\mathbb{E}\left[\int_{t}^{\sigma_{t,a}^{x,\alpha, \pi}}\mathrm{e}^{-q(s-t)}(\mathrm{d}L_{s}-k \,\mathrm{d}I_{s}) \ | \ X_{t}^{t,x,\alpha, \pi}=x\right]\,, \qquad t\geq 0\,,
\end{equation}
where $q>0$ is a discount factor, $k>1$, and $\sigma_{t,a}^{x,\alpha, \pi}$ is the ruin time defined by (\ref{tau}). 
Notice that since $L$ and $I$ are $\mathbb{F}$-predictable and non-decreasing c\`adl\`ag processes, then the integrals of the form $\int_{t}^{u}\mathrm{e}^{-q(s-t)}(\mathrm{d} L_{s}-k\, \mathrm{d}I_{s})$, $u\geq t\geq 0$, exist in the Lebesgue–Stieltjes sense and are progressively measurable (see e.g.~\cite[Proposition 3.4, Chapter I]{jacod2013limit}).
In the sequel, we denote by $\mathbb{E}_x[\cdot]= \mathbb{E}[\cdot \ | \ \ X_{t}^{t,x,\alpha, \pi}=x]$, $t, x\geq 0$ and refer to the function $J$ defined in \eqref{eq:function-J} as the criterion to optimise. The optimisation problem then consists of finding the optimal value function
\begin{equation}\label{v(x)}
    v(t,x) = \sup_{(\alpha,\pi)\in \prod_t(x)}J(t,x,\alpha, \pi), \qquad t, x\in \mathbb{R}_{+}
\end{equation}
and an optimal admissible strategy $(\alpha^{*}, \pi^{*})$ leading to the optimal value function, i.e. a strategy which delivers the maximal return function (\ref{v(x)}). As a result our optimal strategies consist of injecting capital when the size of the deficit below $0$ does not
exceed a certain limit $a$, and paying dividends only when the reserve reaches a certain barrier $b>0$.\\
To solve the optimization problem \eqref{v(x)}, we employ two approaches: the viscosity-solution approach (Section \ref{sec:HJB}) and an analytical approach based on exploiting the characteristic function of the Lévy model under additional structural assumptions on our framework (Section \ref{sec:Guess}). Using the first approach, we do not require the restriction to piecewise constant reinsurance strategies. However, in the second part of the paper, the situation is different. To compute the explicit form of the value function, we need to work with moment-generating functions, as in Equation~(4.1) below, and identify the corresponding Laplace exponent. This step would not be feasible under a fully continuous reinsurance control. Therefore, we consistently assume piecewise constant reinsurance strategies throughout the paper in order to make the analysis tractable and to derive explicit closed-form results. This assumption is also realistic from a practical perspective, since reinsurance contracts are typically adjusted at discrete time intervals. In our model, the insurer and the reinsurer negotiate the reinsurance strategy for a fixed period (e.g., one year or six months), during which it remains unchanged. The reinsurer does not accept instantaneous modifications to the agreed strategy.
\section{The Value Function for the Cram\'er-Lundberg Model}\label{sec:HJB}
In this section, we establish a verification-type result for the stochastic control problem (\ref{v(x)}). We follow the next steps:
\begin{enumerate}
    \item First, we focus on the regularity properties of the value function $v$(lower and upper-bound and Lipschitz-continuity) in Proposition (\ref{Prop1}).
    \item Second, we prove the connection between this value function and the associated partial integral differential equation (of HJB-type) in Theorem (\ref{Thvs}).
    \item Third, we state a comparison principle in which we prove the  uniqueness results for viscosity solutions to the associated HJB equation.
\end{enumerate}
All the results in this section are valid for general laws of the claims $U$. To emphasise this, we consider that the law of $U$ is given by its distribution function denoted by $F$.
\subsection{Some Elementary Properties of the Value Function}
We start with gathering some elementary properties of the value function (\ref{v(x)}) in the following proposition. The proof follows closely similar derivations as in \cite{avram2021equity}  we present it here for completness.  

\begin{Proposition}\label{Prop1}
The following properties hold.
\begin{enumerate}
   \item \label{prop:v-properties1}
For every $t \in \mathbb{R}_{+}$ and $x\in [-a,\infty)$, the set of admissible strategies $\prod_t(x)$ is non-empty.\\
    If $x, x' \in [-a,\infty) \mbox{ such that } x\leq x'$, then $\prod_t(x)\subset \prod_t(x'),\  v(t, x)\leq v(t, x')$, and for every $(\alpha, \pi)\in \prod_t(x),\\  \sigma_{t,a}^{x,\alpha, \pi}\leq  \sigma_{t,a}^{x',\pi,\alpha},\,\, \mathbb{P}-a.s$. 
    \item  \label{prop:v-properties2} For $t\geq 0,\, x \geq -a$, the value function $v(t,x)$ admits the following bounds:
    \begin{enumerate} \label{eq:vbound}
        \item $v(t,x)\geq x^+$,\, where $x^+=\max\{x,0\}$.
        \item $v(t,x)\leq C(1+|x|)$, where $C$ is a positive constant.
    \end{enumerate}
    \item \label{prop:v-properties3} For every $t\geq 0$, the value function is continuous on $[-a,\infty)$.
\end{enumerate}
\end{Proposition}
\begin{proof}
\begin{enumerate}
    \item The classical (no dividend, no injection, no reinsurance) (0, 0, 1) policy is an admissible strategy. The second result is a mere consequence of the definition of $X^{\alpha,\pi}$.
    \item \begin{enumerate}
       \item For $x\geq 0$, we consider the strategy consisting in no capital injection $(I = 0)$, then paying $L_{0}= x$ and (continuously) all the premium $c$ to reinsurance and declare bankruptcy at the ruin time. For this admissible strategy $(\pi^{0},\underline{\alpha})$, we have 
       \begin{align*}
           v(t,x)&\geq J(t,x,\pi^{0},\underline{\alpha})\\           &=x+\mathbb{E}_x\left[\int_{t}^{\tau_{1}}\mathrm{e}^{-q(s-t)} c(\underline{\alpha})\, \mathrm{d} s\right]\\
           &=x\,.
       \end{align*}$$$$
        For $x\in [-a,0)$, by choosing the strategy $(0,0,1)$ until the ruin time, we deduce that $v(t,x)\geq 0$ and so we have $v(t,x)\geq x^+$.
      \item \underline{First case $x\geq 0$:} Let $(\pi,\alpha)$ be an arbitrary admissible strategy. Since $c(\alpha_{s})\leq c$ for all $s\geq0$, then from equation \ref{XLI}, we obtain that 
      $X_{s}^{t,x,\alpha,\pi}\leq x+c(s-t)$. If the insurer’s reserve is $x+c(t-s)$, then, due to discount factor, the optimal dividend strategy is to give $x$ to shareholders at time $t$, followed by a continuous payment of $c$ per unit of time. It shows that $$ J (t,x,\pi,\alpha)\leq x+ \mathbb{E}_{x}\left[\int_{t}^{\infty}\mathrm{e}^{-q(s-t)}c\, \mathrm{d} s\right]=x+ \frac{c}{q}\,.$$
        Taking the supremum over all admissible strategies shows that the value function satisfies: $v(t,x)\leq C(1+x)$, where $C$ is a positive generic constant which could from line to line.\\
        \underline{Second case $x\in [-a,0)$:} 
        from the monotonicity of the value function (Proposition \ref{Prop1} \ref{prop:v-properties1}) and the fact that $v(t,0) \leq C$, it holds that the valued function satisfies $v(t,x) \leq v(t,0) \leq C$.\\
        The two cases show that for fixed $t\geq 0$ and for all $x\geq -a$, we have $v(t,x)\leq C(1+|x|)$, where $C$ is a positive constant.
        \end{enumerate}
    \item We fix $\epsilon>0$, $x\geq -a$ and $t\geq 0$. By monotonicity of the value function, we have: 
    \begin{eqnarray}\label{cad1}
    v(t,x)\leq v(t,x+\epsilon). 
    \end{eqnarray}
We consider a strategy $(\pi,\alpha)\in \prod_t(x+\epsilon)$. We denote by 
$(\pi_{\epsilon},\alpha)$ the strategy $ (L, I+\epsilon,\alpha)$ where the strategy $I+\epsilon$ consists on injecting at time $t$ the amount $\epsilon$ and following the strategy $I$ until ruin. Then 
   $I_s\leq \sum_{i=1}^{N_{s^-}}\alpha_{\tau_i} U_i-(x+\epsilon) $ 
   which implies $I_s + \epsilon \leq \sum_{i=1}^{N_{s^-}}\alpha_{\tau_i} U_i-x $ and $\bigtriangleup L_{s}\leq X_s^{t,x+\epsilon,\alpha,\pi}=X_s^{t,x,\alpha,\pi_{\epsilon}}$, and so $(L, I+\epsilon, \alpha)$ is an admissible strategy.
    It yields that one could modify on $\mathbb{R}_{+}$ an admissible strategy $(\pi,\alpha)\in \prod_t(x+\epsilon)$ into $(\pi_{\varepsilon},\alpha)= (L, I+\epsilon)\in\prod_t(x)$. 
    Then, from the definition of the criterion $J$, and by taking the supremum over all admissible strategies, we obtain, 
    \begin{eqnarray}\label{cad2}
    v(t,x+\epsilon)-k\epsilon\leq v(t,x).
    \end{eqnarray}
    From inequalities (\ref{cad1}) and (\ref{cad2}), we deduce the right continuity of the value function $v$ on $[-a,\infty)$.
    For the left continuity, we fix $\epsilon>0$, $x> -a$ and $t\geq 0$ such that $x-\epsilon >-a$. From the right continuity inequalities (\ref{cad1}) and (\ref{cad2}), we know that for $y>-a$
    \begin{eqnarray}\label{cag1}
    v(t,y)\leq v(t,y+\epsilon)\leq v(t,y)+k\epsilon.
    \end{eqnarray}
    By putting $x=y+\epsilon$, inequality (\ref{cag1}) implies
    \begin{eqnarray}\label{cag2}
    v(t,x)-k\epsilon\leq v(t,x-\epsilon)\leq v(t,x).
    \end{eqnarray}
    This shows the left continuity of $v$ on $(-a,\infty)$ and so $v$ is continuous on $[-a,\infty)$.
\end{enumerate}
\end{proof}
\begin{Remark}
In fact, in the proof of Proposition \ref{Prop1}[\ref{prop:v-properties3}], one can show that for all $y\in [x,x+\epsilon]$, we have $v(t,x)\leq v(t,y)\leq v(t,x)+k\epsilon$ and for all $y\in [x-\epsilon,x]$, we have $v(t,x)-k\epsilon\leq v(t,y)\leq v(t,x)$. This proves that the value function $v$ is uniformly continuous with respect to $x$. Such class of functions is denoted by $UC_x((-a,+\infty))$ for which there is a modulus of continuity $m$ and $r>0$ such that 
\begin{eqnarray*}
    |v(t,x)-v(t,y)|\leq m(|x-y|) \mbox{ for }x,y\in (-a,+\infty), |x-y|\leq r. 
\end{eqnarray*}
We recall that $m$ is a map from $[0,+\infty)$ into $[0,+\infty)$ and satisfy:
\begin{eqnarray*}
&(i)& m \mbox{ is continuous and nondecreasing};\\ 
&(ii)& m(0)=0;\\
&(iii)& m(a+b)\leq m(a)+m(b);\\
&(iv)& m(r)\leq m(1)(1+r).
\end{eqnarray*}
\end{Remark}
\noindent From the latter proposition and remark we infer that the value function $v(t, \cdot)$, for all $t\geq 0$, is positive, monotone increasing, and uniformly continuous. 
\subsection{The HJB System}
Let 
\begin{equation}
G=\{\phi\colon \R_+\times \R_+\rightarrow \R_+ \ | \ \phi(t, \cdot) \in C^1(\R_+), \forall t\geq 0\}.
\end{equation}
Then we define $H: \R_+ \times \R\times G \times \R_+ \rightarrow \R$
as  
\begin{equation}\label{Hm}
   \begin{aligned}
H(t,x,\phi,w) & := \sum_{i}\left[\left(\sup_{\alpha\in [\underline{\alpha}, 1]}\left[c(\alpha)w + \lambda\int_{\mathbb{R}_+} \phi(t,x - \alpha y)\, \mathrm{d} F(y) - (q+\lambda)\phi(t,x)\right]\right)\mathbb{1}_{\{t=t_{i}\}}\right.\\
     &\qquad  + \left.\left(c(\alpha_{t_{i}})w + \lambda\int_{\mathbb{R}_+} \phi(t,x - \alpha_{t_{i}} y)\, \mathrm{d}F(y) - (q+\lambda) \phi(t,x)\right)\mathbb{1}_{\{t\in (t_{i}, t_{i+1})\}} \right],
    \end{aligned}
    \end{equation}
where 
\begin{equation}
    \alpha_{t_i}=\arg\max_{\alpha\in [\underline{\alpha}, 1]} \left[c(\alpha)w + \lambda\int_{\mathbb{R}_+} \phi(t,x - \alpha y)\, \mathrm{d} F(y) - (q+\lambda)\phi(t,x)\right].
\end{equation}
By abuse of notation we use $v'(t,x) =\partial v(t,x)/\partial x$.\\
We will use the Dynamic Programming approach to characterize the value function as viscosity solution of the associated system of HJB equations. In the proof of such characterization, 
we apply the following {\it dynamic programming principle}: 
for any $t\geq 0$ and any stopping time $\tau\geq t$, it holds
\begin{equation}\label{DPP}
   v(t,x) = \sup_{(\alpha, \pi)\in \prod_t(x)}\mathbb{E}_{x}\left[\int_{[t, \tau \wedge\sigma_{t,a}^{x,\alpha, \pi}]} \mathrm{e}^{-q(s-t)}(\mathrm{d} L_{s}-k\, \mathrm{d}I_{s})+\mathrm{e}^{-q(\tau \wedge\sigma_{t,a}^{x,\alpha, \pi}-t)}v(\tau \wedge\sigma_{t,a}^{x,\alpha, \pi},X^{t,x,\alpha,\pi}_{\tau \wedge\sigma_{t,a}^{x,\alpha, \pi}})\right]. 
\end{equation}
The proof of the dynamic principle is quite similar to the one without injections, see e.g.~\cite[lemma 2.2]{azcue2014stochastic}. \\

We consider the following HJB System:

\begin{eqnarray}
    \max\{H(t,x,v ,v'(t,x)), 1 - v'(t,x), v'(t,x)-k\} =0, & x \in (0,\infty),\label{HJB1}\\
    \max \{H(t,x,v ,v'(t,x)), v'(t,x) - k\}  =0\,, &  x \in (-a, 0),\label{HJB2} \\
    v(t,x)  =0\,, &  x \in (-\infty,-a],
\end{eqnarray}
where $H$ is as in \eqref{Hm}. The lower bound $1$ on the partial 
derivative of $v$ is linked to the possibility of lump-sum dividend payments, and the $k$ upper bound
is linked to the possibility of lump-sum instant reserves replenishment. 
The idea is to find the supremum of $\alpha\in [\underline{\alpha}, 1]$ at time  $t=t_{i}$ and use this supremum during the period $[t_{i}, t_{i+1})$.\\
In the next theorem, we will use the powerful tool in stochastic control named viscosity solutions. we show that for any time $t \geq 0$, the
value function $v$ in (\ref{v(x)}) is  a viscosity solution to \eqref{HJB1}-\eqref{HJB2}. We use the notion of continuous viscosity solutions defined on $(0,\infty)$ as follows:
\begin{definition}\label{def:test-function} We fix $t\geq 0$. Let $u(t,.): (-a,0)\cup(0, \infty) \to \mathbb{R}$ be a uniform continuous function. 
\begin{itemize}
    \item $u$ is a viscosity supersolution of (\ref{HJB1})-(\ref{HJB2})
    at \( x_{0} \in (-a,0)\cup (0, \infty) \), if any function \( \psi: (-a,0)\cup (0,\infty) \to \mathbb{R} \) continuously differentiable in the second variable such that
    \( u - \psi \) reaches the minimum at \( x_{0} \) satisfies
    \begin{eqnarray*}
    \max\{H(t,x_{0},u(t,x_{0}),\psi'(x_{0})), 1 - \psi'(x_{0}),  \psi'(x_{0})-k\}  &\leq& 0,\,\mbox{if }x_0 \in (0,\infty)\\
    \max \{H(t,x_{0},u(t,x_{0}),\psi'(x_{0})),  \psi'(x_{0})-k\}&\leq& 0,\,\mbox{if }x_0 \in (-a,0).
    \end{eqnarray*}
    \item $u$ is a viscosity subsolution of (\ref{HJB1})-(\ref{HJB2}) at \( x_{0} \in (-a,0)\cup(0, \infty) \)
    if any function \( \psi: (-a,0)\cup (0, \infty)  \to \mathbb{R} \) continuously differentiable in the second variable such that \( u - \psi \)
    reaches the maximum at \( x_{0} \) satisfies
    \begin{eqnarray*}
    \max\{H(t,x_{0},u(t,x_{0}),\psi'(x_{0})), 1 - \psi'(x_{0}),  \psi'(x_{0})-k\}  &\geq& 0,,\,\mbox{if }x_0 \in (0,\infty)\\
     \max \{H(t,x_{0},u(t,x_{0}),\psi'(x_{0})),  \psi'(x_{0})-k\}&\geq& 0,\,\mbox{if }x_0 \in (-a,0).
    \end{eqnarray*}
    \item We say that $u$ is a continuous viscosity solution of (\ref{HJB1})-(\ref{HJB2}) at \( x_{0} \in (-a,0)\cup(0, \infty) \) if it is both a subsolution and a supersolution of (\ref{HJB1})-(\ref{HJB2}).
\end{itemize}
\end{definition}
\noindent We use a similar definition for the viscosity solution on $(-a,0)$. \\
The HJB system \eqref{HJB1}--\eqref{HJB2} is defined on open sets, namely the open intervals $(0,\infty)$ and $(-a,0)$. We made this choice because the viscosity solutions are naturally defined on open sets. Alternatively, one could use the concept of constrained viscosity solutions defined on closed sets, which is more delicate. However, in both cases, the comparison theorem is available only on open sets, which justifies our approach.
\begin{theorem}\label{Thvs}
We fix $t\geq 0$. The value function $v(t, \cdot)$ as defined in (\ref{v(x)}) is a continuous viscosity solution of (\ref{HJB1})-(\ref{HJB2}) in $(-a,0)\cup(0,\infty)$.
\end{theorem}
\begin{proof}
We first prove that $v(t, \cdot)$ is a viscosity {\it super-solution} of (\ref{HJB1}) in $(0, \infty)$, where $t\in [t_i,t_{i+1})$ is fixed.
Given an initial surplus $x> 0$ and any $\delta > 0 $, we consider the admissible strategy $(\pi_s, \alpha_s) := (\delta s, 0, \alpha)$, $s\geq 0$ (which pays dividends at constant rates $\delta$, there is no capital injection, and we choose a constant proportional reinsurance strategy $\alpha \in [\underline \alpha,1]$).
The corresponding controlled surplus process is
\begin{equation*}
    X_{s}^{t,x,\alpha,\pi}= X_{s}^{t,x,\alpha}-\delta (s-t)\,, \qquad 0\leq t\leq s\,.
\end{equation*}
We denote by $\tau_{1}$ the first claim. One could choose $u$ such that $t_{i+1}> u >t\geq t_{i}$ and $ \tau_{1} \wedge u \leq  \sigma_{t,a}^{x,\alpha, \pi}$.   
With the particular choice of $(\pi, \alpha) := (\delta s, 0 , \alpha)$, the dynamic programming principle \eqref{DPP} implies that,
\begin{align}\label{eq:value-funtion-t-i}
    v(t,x) &\geq \mathbb{E}_{x}\left[\mathrm{e}^{-q(u \wedge \tau_{1}-t )}v(u \wedge \tau_{1},X^{t,x,\alpha,\pi}_{u \wedge \tau_{1}})\right]+\mathbb{E}_{x}\left[\int_{t}^{u \wedge \tau_{1}}\mathrm{e}^{-q(s-t)} \mathrm{d} L_{s}\right]  \nonumber\\
    &=\mathrm{e}^{-q(u-t)}v(t,x+(c(\alpha)-\delta)(u-t))\mathbb{P}_{x}(\tau_{1}>u)+\mathbb{E}_{x}\left[\mathbb{1}_{\tau_{1}\leq u} \mathrm{e}^{-q(\tau_{1}-t)}v(\tau_{1}, X^{t, x,\alpha,\pi}_{\tau_{1}}) \right] \nonumber\\
    &\qquad +\mathbb{E}_{x}\left[\int_{t}^{u}\mathrm{e}^{-q(s-t)}\delta\, \mathrm{d} s \ \ \mathbb{1}_{\tau_{1}> u}\right] +\mathbb{E}_{x}\left[\int_{t}^{\tau_{1}}\mathrm{e}^{-q(s-t)}\delta\, \mathrm{d} s \ \  \mathbb{1}_{\tau_{1}\leq u}\right]\nonumber\\
    &\geq \mathrm{e}^{-q(u-t)}v(t,x+(c(\alpha)-\delta)(u-t))\mathbb{P}_{x}(\tau_{1}>u)+\mathbb{E}_{x}\left[\mathbb{1}_{\tau_{1}\leq u} \mathrm{e}^{-q(\tau_{1}-t)}\right] \int_{\mathbb{R}_{+}} v(t,x - \alpha y)\,\mathrm{d}F(y) \nonumber\\
    &\qquad +\frac{1}{q} \delta(1-\mathrm{e}^{-q(u-t)})\mathbb{P}_{x}(\tau_{1}>u) +\frac{1}{q} \delta \mathbb{E}_{x}\left[(1-\mathrm{e}^{-q(\tau_{1}-t)}) \mathbb{1}_{\tau_{1}\leq u}\right]\nonumber\\
    &=\mathrm{e}^{-(\lambda+q)(u-t)}v(t,x+(c(\alpha)-\delta)(u-t))\nonumber\\\
    &\qquad +\frac{1}{\lambda+q}(1-\mathrm{e}^{-(\lambda+q)(u-t)})\left[\lambda \int_{\mathbb{R}_{+}} v(t,x - \alpha y)\, \mathrm{d}F(y) +\delta \right],
\end{align}
where in the third inequality, we used that for $t_i \leq t \leq\tau_1\leq u$,
\begin{align*}
    X_{\tau_1}^{t,\alpha,\pi} &=x + c(\alpha)(\tau_1-t) - \alpha U_1\\
    &\geq x-\alpha U_1,
\end{align*}
that the value function $v(t, \cdot)$ is increasing for all $t\geq 0$ (see property \ref{prop:v-properties1} in Proposition \ref{Prop1}), and that $\tau_i$ and $U$ are independent for every $i \in \mathbb{N}$. By passing $v(t,x)$ on the right-hand side of \eqref{eq:value-funtion-t-i} and dividing by $u-t$, we infer
\begin{align}\label{eq:bigger-than-0}
        0 &\geq \frac{\mathrm{e}^{-(\lambda+q)(u-t)}\left(v(t,x+(c(\alpha)-\delta)(u-t))-v(t,x)\right)+\mathrm{e}^{-(\lambda+q)(u-t)} v(t,x)-v(t,x)}{u-t}\nonumber\\
        &\qquad +\frac{1}{\lambda+q} \frac{1-\mathrm{e}^{-(\lambda+q)(u-t)}}{u-t} \left[\lambda \int_{\mathbb{R}_{+}} v(t,x - \alpha y)\, \mathrm{d}F(y) +\delta\right]\nonumber\\
        &= \frac{\mathrm{e}^{-(\lambda+q)(u-t)}\left(v(t,x+(c(\alpha)-\delta)(u-t))-v(t,x)\right)}{u-t}+\frac{(\mathrm{e}^{-(\lambda+q)(u-t)} -1)}{u-t}v(t,x)\nonumber \\
        &\qquad +\frac{1}{\lambda+q} \frac{1-\mathrm{e}^{-(\lambda+q)(u-t)}}{u-t} \left[\lambda \int_{\mathbb{R}_{+}} v(t,x - \alpha y)\, \mathrm{d}F(y) +\delta\right]\,.
\end{align}
Let $\psi \in G$ be such that $\psi(t,x)=\psi(t_{i},x), \forall  t \in[t_{i},t_{i+1})$ and $\forall x \in \R_+$. Moreover, let $x_{0}$ be fixed in $(0,\infty)$ be a strict global minimiser of $v-\psi$, i.e., 
\begin{equation*}
    0=(v-\psi)(t,x_{0})=\min_{x\in(0,\infty)}(v-\psi)(t,x), \ \  \mbox{for all } t\geq 0.
\end{equation*}
We get from \eqref{eq:bigger-than-0}
\begin{align}\label{eq:psi-0}
        0 &\geq \frac{\mathrm{e}^{-(\lambda+q)(u-t)}\left(\psi(t,x_0+(c(\alpha)-\delta)(u-t))-\psi(t,x_0)\right)}{u-t}+\frac{(\mathrm{e}^{-(\lambda+q)(u-t)} -1)}{u-t}v(t,x_0)\nonumber\\
        &\qquad +\frac{1}{\lambda+q} \frac{1-\mathrm{e}^{-(\lambda+q)(u-t)}}{u-t} \left[\lambda \int_{\mathbb{R}_{+}} v(t,x_0 - \alpha y)\, \mathrm{d}F(y) +\delta\right]\,.
\end{align}
Taking $u \downarrow t$ in \eqref{eq:psi-0}, we deduce
\begin{equation}\label{eq:HJB-psi}
    0\geq c(\alpha)\psi'(t,x_0) + \lambda\int_{\mathbb{R}_+} v(t,x_0 - \alpha y)\, \mathrm{d} F(y) - (q+\lambda)v(t,x_0)+\delta(1-\psi'(t,x_0)), 
\end{equation}
for any $x_0 \in (0, \infty)$, $t \in [t_i, t_{i+1})$, and  $\delta\geq 0$. Taking $\delta=0$ in \eqref{eq:HJB-psi}, we infer
\begin{equation}\label{HSupert}
    0\geq c(\alpha)\psi'(t,x_0) + \lambda\int_{\mathbb{R}_+} v(t,x_0 - \alpha y)\, \mathrm{d} F(y) - (q+\lambda)v(t,x_0),
\end{equation}
Since \eqref{HSupert} holds true for any $\alpha\in [\underline{\alpha}, 1]$ we obtain for $t=t_i$,
\begin{equation*}
    0\geq \sup_{\alpha\in [\underline{\alpha}, 1]} \left\{ c(\alpha)\psi'(t_{i},x_0) + \lambda\int_{\mathbb{R}_+} v(t_{i},x_0 - \alpha y)\, \mathrm{d} F(y) - (q+\lambda)v(t_{i},x_0)\right\},
\end{equation*}
and for any $t \in (t_i, t_{i+1})$, we get
\begin{equation}\label{super1}
    0\geq c(\alpha_{t_{i}})\psi'(t,x_0) + \lambda\int_{\mathbb{R}_+} v(t,x_0 - \alpha_{t_{i}} y)\, \mathrm{d} F(y) - (q+\lambda)v(t,x_0),
\end{equation}
where 
\begin{equation}
    \alpha_{t_i}=\displaystyle \arg \max_{\alpha\in [\underline{\alpha}, 1]} \left[c(\alpha)w + \lambda\int_{\mathbb{R}_+} \phi(t_i,x - \alpha y)\, \mathrm{d} F(y) - (q+\lambda)\phi(t_i,x)\right].
\end{equation}
Moreover, since \eqref{eq:HJB-psi} is valid for any $x_0 \in (0, \infty)$, $\alpha\in [\underline{\alpha}, 1]$, then sending $\delta$ to infinity, we obtain 
\begin{equation}\label{super2}
1-\psi'(t,x)\leq 0.
\end{equation}
Furthermore, we infer that 
\begin{equation}\label{super3}
\psi'(t,x_0)-k \leq 0.
\end{equation}
by observing that
\begin{equation*}
 \psi(t,x_0+\varepsilon)-\psi(t,x_0) \leq  v(t,x_0+\varepsilon)-v(t,x_0)\leq k\varepsilon.
\end{equation*}
From inequalities (\ref{super1}),(\ref{super2}) and (\ref{super3}), we deduce that $v(t, \cdot)$ is a viscosity {\it super-solution} of (\ref{HJB1}) in $(0, \infty)$.\\
Next we prove that $v(t, \cdot)$ is a viscosity sub-solution of (\ref{HJB1}) in $(0, \infty)$. 
Let $\psi \in G$ be such that $\psi(t,x)=\psi(t_{i},x),\mbox{ for fixed }  t \in[t_{i},t_{i+1})$ and $\forall x \in \R_+$. Moreover, let $x_{0} \in (0,\infty)$ be a strict global minimiser of $v-\psi$, i.e., 
\begin{equation*}
    0=(v-\psi)(t,x_{0})=\max_{x\in(0,\infty)}(v-\psi)(t,x).
\end{equation*}
We aim to show that
\begin{equation}\label{HJBsubspsi}
\max\{H(t,x_{0},v(t,x_{0}),\psi'(t,x_{0})), 1 - \psi'(t,x_{0}),  \psi'(t,x_{0})-k\}  \geq 0.
\end{equation}
Let $x_0, \zeta \geq 0$. Define $B(x_0, \zeta) = \{x, |x-x_0|\leq \zeta\}$.
Suppose that (\ref{HJBsubspsi}) does not hold. Then the left-hand side of (\ref{HJBsubspsi}) is negative. By smoothness of $\psi$ and since $x \rightarrow H(t, x, \cdot,\cdot)$ is continuous, we know there exists $\zeta>0$ and $\xi>0$ and $\varepsilon>0$ satisfying
\begin{equation}\label{HJBsubsinf}
\max\{H(t,x,v,\psi'(t,x)), 1 - \psi'(t,x), \psi'(t,x)-k\}  < - q\varepsilon, 
\end{equation}
for all $x \in B(x_{0},\zeta)$ and $t \in [t_i, t_{i+1})$ and
\begin{equation}\label{inB}
    v(t,x')\leq \psi(t,x')-\xi,
\end{equation}
where $x'=x_{0}\pm\zeta$. Let $(\alpha, \pi)\in \prod_t(x)$ and 
\begin{equation*}
    \tau=\inf \{s\geq t, \  X_{s}^{t,x_{0},\alpha,\pi} \notin B(x_{0},\zeta)\}.
\end{equation*}
It holds that $\tau$ is an $\mathbb{F}$-stopping time (see e.g.~\cite[Chapter I, Proposition 4.6]{revuz2013continuous}) and by right continuity of $X_{.}^{t,x_{0},\alpha,\pi}$, we have $\tau>t$, a.s. 
We truncate $\tau$ by a fixed constant $T>0$ in order to make it finite.
We introduce
\begin{equation*}
\hat{\tau}=\tau\land T.
\end{equation*}
On the set $\{\hat{\tau}=\tau \}$, using that $v$ is non-decreasing and $\psi'$ is positive, we get from (\ref{inB})
\begin{equation*}
    v(\hat{\tau},X_{\hat{\tau}}^{t,x_{0},\alpha,\pi})\leq v(\hat{\tau},x_{0}-\zeta)\leq \psi(\hat{\tau},x_{0}-\zeta)-\xi \leq \psi(\hat{\tau},X_{\hat{\tau}^{-}}^{t,x_{0},\alpha,\pi}) -\xi\,.
\end{equation*}
On the set $\{\hat{\tau}=T\}$, we have:
\begin{equation*}
    v(\hat{\tau},X_{\hat{\tau}}^{t,x_{0},\alpha,\pi}) \leq \psi(\hat{\tau},X_{\hat{\tau}^{-}}^{t,x_{0},\alpha,\pi})\,.
\end{equation*}
Let $L^c$ and $I^c$ stand for the continuous part in the decomposition of $L$ and $I$. Applying It\^o's Lemma to $\mathrm{e}^{-q(s-t)} \psi(s,X_{s}^{t,x_{0},\alpha,\pi})$ on $[t,\hat{\tau}]$, we get, 
\begin{align}
    &\mathrm{e}^{-q(\hat{\tau}-t)} v(\hat{\tau},X_{\hat{\tau}}^{t,x_{0},\alpha,\pi})\nonumber\\
    &\leq \mathrm{e}^{-q(\hat{\tau}-t)}\psi(\hat{\tau}^{-},X_{\hat{\tau}^{-}}^{t,x_{0},\alpha,\pi})-\mathrm{e}^{-q(\hat{\tau}-t)}\xi \mathbb{1}_{\{\hat{\tau}=\tau\}}\nonumber\\
    &=\psi(t,x_{0})+\int_{t}^{\hat{\tau}}(-q\mathrm{e}^{-q(s-t)} \psi(s,X_{s}^{t,x_{0},\alpha,\pi})+\mathrm{e}^{-q(s-t)}\psi'(s,X_{s}^{t,x_{0},\alpha,\pi})c(\alpha_{s}))\, \mathrm{d}s\nonumber\\
    &\qquad +\sum_{t\leq s\leq \hat{\tau}}\mathrm{e}^{-q(s-t)}(\psi(s,X_{s}^{t,x_{0},\alpha,\pi})-\psi(s,X_{s^{-}}^{t,x_{0},\alpha,\pi}))\nonumber\\
    &\qquad+\int_{t}^{\hat{\tau}}\mathrm{e}^{-q(s-t)}\psi'(s,X_{s^{-}}^{t,x_{0},\alpha,\pi})(-\mathrm{d} L_{s}^{c}+\mathrm{d} I_{s}^{c})-\mathrm{e}^{-q(\hat{\tau}-t)}\xi \mathbb{1}_{\{\hat{\tau}=\tau\}}\nonumber,
\end{align}
which implies 
\begin{align}\label{ItoSub}
    &\mathrm{e}^{-q(\hat{\tau}-t)} v(\hat{\tau},X_{\hat{\tau}}^{t,x_{0},\alpha,\pi})\nonumber\\
    &\leq \psi(t,x_{0})+\int_{t}^{\hat{\tau}}(-q\mathrm{e}^{-q(s-t)} \psi(t_{i},X_{s}^{t,x_{0},\alpha,\pi})+\mathrm{e}^{-q(s-t)}\psi'(t_{i},X_{s}^{t,x_{0},\alpha,\pi})c(\alpha_{s}))\, \mathrm{d}s\nonumber\\
    &\qquad+\int_{t}^{\hat{\tau}}\int_{\R_+}(\psi(t_{i},X_{s^{-}}^{t,x_{0},\alpha,\pi}-\alpha_{s}z)-\psi(t_{i},X_{s^{-}}^{t,x_{0},\alpha,\pi}))\mu(\mathrm{d}s,\mathrm{d}z)\nonumber\\
    &\qquad+\int_{t}^{\hat{\tau}}\mathrm{e}^{-q(s-t)}\psi'(t_{i},X_{s^{-}}^{t,x_{0},\alpha,\pi})(-\mathrm{d}L_{s}^{c}+\mathrm{d}I_{s}^{c})\nonumber\\
&\qquad+\sum_{ \begin{subarray}{c}
    {L_{s}\ne L_{s-},} \\
    {t\leq s\leq \hat{\tau}}
    \end{subarray}}\mathrm{e}^{-q(s-t)}(\psi(t_{i},X_{s^{-}}^{t,x_{0},\alpha,\pi}-\triangle L_{s})-\psi(t_{i},X_{s^{-}}^{t,x_{0},\alpha,\pi}))\nonumber\\
    &\qquad+\sum_{\begin{subarray}{c}
    {I_{s}\ne I_{s-},} \\
    {t\leq s\leq \hat{\tau}}
    \end{subarray}}\mathrm{e}^{-q(s-t)}(\psi(t_{i},X_{s^{-}}^{t,x_{0},\alpha,\pi}+\triangle I_s)-\psi(t_{i},X_{s^{-}}^{t,x_{0},\alpha,\pi})).
\end{align}
As for all $x\in B(x_{0},\zeta)$ we have
\begin{equation}\label{contratdictionppd1}
    H(t,x,v,\psi'(t,x))  \leq -q\varepsilon,
\end{equation}
\begin{equation}\label{contratdictionppd2}
    1 - \psi'(t,x)  \leq 0\,,
\end{equation}
and
\begin{equation}\label{contratdictionppd3}
    \psi'(t,x) - k  \leq 0\,.
\end{equation}
From inequality (\ref{contratdictionppd1}), it follows that
\begin{align}\label{HSub}
&c(\alpha_{s}) \psi'(t_{i},X_{s}^{t,x_{0},\alpha,\pi})-q\psi(t_{i},X_{s}^{t,x_{0},\alpha,\pi})+\int_{\R_+}(\psi(t_{i},X_{s_{-}}^{t,x_{0},\alpha,\pi}-\alpha_{s}z)-\psi(t_{i},X_{s_{-}}^{t,x_{0},\alpha,\pi}))\nu(\mathrm{d}z)\nonumber\\
&\qquad \leq -q\varepsilon.
\end{align}
Using the mean value theorem and inequalities (\ref{contratdictionppd2}) and (\ref{contratdictionppd3}), we deduce that 
\begin{align}\label{LSub}
    &-\int_{t}^{\hat{\tau}}\mathrm{e}^{-q(s-t)}\psi'(t_{i},X_{s_{-}}^{t,x_{0},\alpha,\pi})\,\mathrm{d}L_{s}^{c}+\sum_{ \begin{subarray}{c}
    {L_{s}\ne L_{s-},} \\
    {t\leq s\leq \hat{\tau}}
    \end{subarray}}\mathrm{e}^{-q(s-t)}(\psi(t_{i},X_{s^{-}}^{t,x_{0},\alpha,\pi}-\triangle L_{s})-\psi(t_{i},X_{s^{-}}^{t,x_{0},\alpha,\pi}))\nonumber\\
    &\qquad \leq -\int_{t}^{\hat{\tau}}\mathrm{e}^{-q(s-t)}\,\mathrm{d}L_{s}
\end{align}
and 
\begin{align}\label{ISub}
    &\int_{t}^{\hat{\tau}}\mathrm{e}^{-q(s-t)}\psi'(t_{i},X_{s_{-}}^{t,x_{0},\alpha,\pi})\,\mathrm{d}I_{s}^{c}+\sum_{\begin{subarray}{c}
    {I_{s}\ne I_{s-},} \\
    {t\leq s\leq \hat{\tau}}
    \end{subarray}}\mathrm{e}^{-q(s-t)}(\psi(t_{i},X_{s^{-}}^{t,x_{0},\alpha,\pi}+\triangle I_s)-\psi(t_{i},X_{s^{-}}^{t,x_{0},\alpha,\pi}))\nonumber\\
    &\qquad \leq \int_{t}^{\hat{\tau}}\mathrm{e}^{-q(s-t)}k\, \mathrm{d}I_{s}\,.
\end{align}
Substituting (\ref{HSub}), (\ref{LSub}) and (\ref{ISub}) into (\ref{ItoSub}), we obtain,
\begin{align}\label{eq:bounds-v}
    &\mathrm{e}^{-q(\hat{\tau}-t)} v(\hat{\tau},X_{\hat{\tau}}^{t,x_{0},\alpha,\pi})\nonumber\\
    &\quad \leq\psi(t,x_{0})-\int_{t}^{\hat{\tau}}\mathrm{e}^{-q(s-t)}(\mathrm{d} L_{s}-k\, \mathrm{d} I_{s})-\int_{t}^{\hat{\tau}}\mathrm{e}^{-q(s-t)} q \varepsilon \, \mathrm{d} s-\mathrm{e}^{-q(\hat{\tau}-t)}\xi \mathbb{1}_{\{\hat{\tau}=\tau\}}.
\end{align}
Taking the expectation in \eqref{eq:bounds-v}, we obtain
\begin{align}\label{cntrappd}
        \psi(t,x_{0}) &\geq \mathbb{E}_{x_{0}}\left[\int_{t}^{\hat{\tau}} \mathrm{e}^{-q(s-t)}(\mathrm{d} L_{s}-k\, \mathrm{d} I_{s})+\mathrm{e}^{-q(\hat{\tau}-t)}v(\hat{\tau},X_{\hat{\tau}}^{t,x_{0},\alpha,\pi})\right]\nonumber\\
        &\qquad +\mathbb{E}_{x_{0}}\left[\int_{t}^{\hat{\tau}}\mathrm{e}^{-q(s-t)} q \varepsilon \, \mathrm{d} s+\mathrm{e}^{-q(\hat{\tau}-t)}\xi \mathbb{1}_{\{\hat{\tau}=\tau\}}\right]\nonumber\\ 
        &\geq \mathbb{E}_{x_{0}}\left[\int_{t}^{\hat{\tau}} \mathrm{e}^{-q(s-t)}(\mathrm{d} L_{s}-k\, \mathrm{d} I_{s})+\mathrm{e}^{-q(\hat{\tau}-t)}v(\hat{\tau},X_{\hat{\tau}}^{t,x_{0},\alpha,\pi})\right]\nonumber\\
        &\qquad +\mathbb{E}_{x_{0}}\left[\int_{t}^{T}\mathrm{e}^{-q(s-t)} q \varepsilon \, \mathrm{d} s{\bf 1}_{\{\hat{\tau}=T\}}+\mathrm{e}^{-q(T-t)}\xi \mathbb{1}_{\{\hat{\tau}=\tau\}}\right]\nonumber\\
        &\geq \mathbb{E}_{x_{0}}\left[\int_{t}^{\hat{\tau}} \mathrm{e}^{-q(s-t)}(\mathrm{d} L_{s}-k\, \mathrm{d} I_{s})+\mathrm{e}^{-q(\hat{\tau}-t)}v(\hat{\tau},X_{\hat{\tau}}^{t,x_{0},\alpha,\pi})\right]\nonumber\\
        &\qquad +\varepsilon \left(1-\mathrm{e}^{-q(T-t)}\right)\mathbb{P}({\hat{\tau}}=T)+\xi\mathrm{e}^{-q(T-t)}\mathbb{P}({\hat{\tau}}=\tau)
\end{align}
As for any $(\alpha, \pi)\in \prod(x)$, we have $\varepsilon \left(1-\mathrm{e}^{-q(T-t)}\right)\mathbb{P}({\hat{\tau}}=\tau)+\xi\mathrm{e}^{-q(T-t)}\mathbb{P}({\hat{\tau}}=T) >0$, we deduce from inequality (\ref{cntrappd}) that 
\begin{align}\label{}
     v(t,x_{0})=\psi(t,x_{0}) > \sup_{(\alpha, \pi)\in \prod(x)}\mathbb{E}_{x_{0}}\left[\int_{t}^{\hat{\tau}} \mathrm{e}^{-q(s-t)}(\mathrm{d}L_{s}-k\, \mathrm{d}I_{s})+\mathrm{e}^{-q(\hat{\tau}-t)}v(\hat{\tau},X_{\hat{\tau}}^{t,x_{0},\alpha,\pi})\right]\nonumber.
\end{align}
which contradicts the Dynamic Programming Principle \eqref{DPP}.\\
For $x\in (-a, 0)$, we prove similarly that the value function is a viscosity solution of the HJB  (\ref{HJB2}).\\

\end{proof}

\subsection{Comparison Principle} 
In the next theorem, we state a comparison principle
from which follows the uniqueness of the viscosity solution to \eqref{HJB1}-\eqref{HJB2} on $(-a,0)\cup(0, \infty)$. 
\begin{theorem}\label{THC} 
We fix $t\geq 0$. Let $v_{1}(t, \cdot)$ and $v_{2}(t, \cdot)$ be respectively a viscosity super-solution and a viscosity sub-solution of (\ref{HJB1}) on $(0,\infty)$. We assume that $v_1$ and $v_2$ are uniformly continuous, satisfy a linear growth condition and 
\begin{eqnarray}\label{civ1v2}
v_{1}(t,0)\geq v_{2}(t,0).
\end{eqnarray}
Then it holds 
\begin{equation*}
    v_{1}(t,x)\geq v_{2}(t,x),\ \ \ \forall x\in (0, \infty)\,.
\end{equation*}
\end{theorem}
\begin{proof}
For $\varepsilon, \beta>0$, we define
$\Psi:[0,\infty)^{3}\rightarrow\mathbb{R}\cup\{-\infty\}$ as
\begin{equation*}
    \Psi(t,x,y)=v_{2}(t,x)-v_{1}(t,y)-\frac{1}{\varepsilon}|x-y|^{2}-\beta(|x|^{2}+|y|^{2})\,.
\end{equation*}
Since $\Psi(t, \cdot, \cdot)$ is continuous with respect to the second and third variables and $\lim_{|x|,|y|\rightarrow \infty} \Psi(t,x,y)=-\infty$, then $\Psi$ has a global maximum at the point $(t,x_{0},y_{0}) \in [0,\infty)^{3} $ and it holds
\begin{equation}\label{psimax}
    \Psi(t,x_{0},y_{0})\geq \Psi(t,x,y)\,, \ \ \forall  x,y \in [0,\infty)^{2}.
\end{equation}
In particular, $\Psi(t,x_{0},x_{0})+\Psi(t,y_{0},y_{0})\leq 2\Psi(t,x_{0},y_{0}).$ Hence 
\begin{equation}\label{eq:bound1}
\frac{2}{\varepsilon}|x_{0}-y_{0}|^{2}\leq v_{2}(t,x_{0})-v_{2}(t,y_{0})+v_{1}(t,x_{0})-v_{1}(t,y_{0}).
\end{equation}
From \eqref{eq:bound1} and the fact that $v_{1}(t, \cdot)$ and $v_{2}(t, \cdot)$ are uniformly continuous, we infer
\begin{equation}\label{majeps}
    |x_{0}-y_{0}|\leq k'\sqrt{\varepsilon}, 
\end{equation}
for $k'$ being a positive constant. Taking $x=y=0$ in (\ref{psimax}), we get
$$\beta(|x_{0}|^{2}+|y_{0}|^{2})\leq C (|x_{0}|+|y_{0}|).$$
From the latter we deduce that there exists a positive constant $C_{0}$ independent of $\beta$ such that 
\begin{equation}
    \beta|x_{0}|\leq C_{0}\,, \quad \mbox{   and   } \quad  \beta|y_{0}|\leq C_{0}.
\end{equation}
Since $x\rightarrow \Psi(t,x,y_{0}) $ takes its maximum at $x=x_{0}$, we get
\begin{equation*}
    v_{2}(t,x_{0})-\frac{1}{\varepsilon}|x_{0}-y_{0}|^{2}-\beta|x_{0}|^{2}\geq v_{2}(t,x)-\frac{1}{\varepsilon}|x-y_{0}|^{2}-\beta|x|^{2}.
\end{equation*}
We put $\phi_{2}(t,x)=\frac{1}{\varepsilon}|x-y_{0}|^{2}+\beta|x|^{2}$. It holds $\phi_{2}(t, \cdot) \in C^1((0,\infty))$, for all $t\geq 0$ and we have
\begin{equation*}
    (v_{2}-\phi_{2})(t,x_{0})=\max_{x \in [0, \infty)}(v_{2}-\phi_{2})(t,x).
\end{equation*}
If $x_{0}=0$, we have $\Psi(t,x,x)\leq \Psi(t,0,y_{0})$
\begin{equation*}
    v_{2}(t,x)-v_{1}(t,x)-2\beta|x|^{2}\leq v_2(t,0)-v_{1}(t,y_{0})-\frac{1}{\varepsilon}|y_{0}|^{2}-\beta|y_{0}|^{2}\leq v_2(t,0)-v_{1}(t,y_{0}).
\end{equation*}
Sending $\varepsilon$ to $0$ and using inequality (\ref{majeps}), the continuity of $v_1$ and inequality (\ref{civ1v2}),
we deduce that 
\begin{equation*}
    v_{2}(t,x)-v_{1}(t,x)-2\beta|x|^{2}\leq 0.
\end{equation*}
Taking $\beta \downarrow 0$, one gets
\begin{equation}\label{eq:inequality-v_2-smaller}
    v_{2}(t,x)\leq v_{1}(t,x), \ \ \ \forall x\geq 0\,.
\end{equation}
If $y_{0}=0$ or $x_{0}=y_{0}=0$, similar derivations yield again the inequality \eqref{eq:inequality-v_2-smaller}.\\
It remains to study the case when $x_{0}\ne 0$ and $y_{0}\ne 0$. 
By the  definition of viscosity sub-solution of (\ref{HJB1}), we have 
\begin{equation}\label{HJBsubs}
\max\{H(t,x_{0},v_{2},\phi_{2}'(t,x_{0})), 1 - \phi_{2}'(t,x_{0}), \phi_{2}'(t,x_{0})-k\}  \geq 0.
\end{equation}
Likewise, Since $y\rightarrow \Psi(t,x_{0},y) $ takes its maximum at $y=y_{0}$, we have
\begin{equation*}
    v_{1}(t,y_{0})+\frac{1}{\varepsilon}|x_{0}-y_{0}|^{2}+\beta|y_{0}|^{2}\leq v_{1}(t,y)+\frac{1}{\varepsilon}|x_{0}-y|^{2}+\beta|y|^{2}.
\end{equation*}
For $t\geq 0$, we put $\phi_{1}(t,y)=-\frac{1}{\varepsilon}|x_{0}-y|^{2}-\beta|y|^{2}$. Then $\phi_{1}(t, \cdot)$ is $C^{1}((0,\infty))$ and  
\begin{equation*}
    (v_{1}-\phi_{1})(t,y_{0})=\min_{y\in (0,\infty)}(v_{1}-\phi_{1})(t,y)\,.
\end{equation*}
By the  definition of viscosity super-solution of (\ref{HJB1}), we have 
\begin{equation}\label{HJBsupers}
\max\{H(t,y_{0},v_{1},\phi_{1}'(t,y_{0})), 1 - \phi_{1}'(t,y_{0}),  \phi_{1}'(t,y_{0})-k\}  \leq 0.
\end{equation} 
Inequalities \eqref{HJBsubs} and \eqref{HJBsupers} imply
\begin{align*}
    &\max\{H(t,y_{0},v_{1},\phi_{1}'(t,y_{0})), 1 - \phi_{1}'(t,y_{0}),  \phi_{1}'(t,y_{0})-k\}\\
    &\qquad -\max\{H(t,x_{0},v_{2},\phi_{2}'(t,x_{0})), 1 - \phi_{2}'(t,x_{0}),  \phi_{2}'(t,x_{0})-k\}  \leq 0.
\end{align*}
Observing that $\max\{a_{1},a_{2}, a_{3}\} -\max\{b_{1},b_{2}, b_{3}\}  \leq 0$, implies $\exists \  j \in \{1,2,3\}$, such that $\forall i \in \{1,2,3\}$, we have $a_{i}\leq b_{j}$, which could be formulated by, $ \exists j \in \{1,2,3\},$ such that $b_{j}=\max\{a_{1},a_{2}, a_{3}, b_{1},b_{2}, b_{3}\}$.

\noindent First, notice that $\phi_{2}'(t,x_{0})-k$, can not be the maximum, because if it was, we would have:
$\phi_{2}'(t,x_{0})-k-(1 - \phi_{1}'(t,y_{0}))\geq 0$, which implies
$(2/\varepsilon +\beta)(x_{0}-y_{0})\geq (k+1)/2$.
From inequality (\ref{majeps}), we obtain $(2/\varepsilon +\beta)k^{'}\sqrt{\varepsilon}\geq (k+1)/2$. Taking $\beta= \frac{1}{\varepsilon} $ and sending $\varepsilon \to \infty$, we get  $0\geq (k+1)/2$, which is absurd. Hence we only have two cases: 
\begin{enumerate}
    \item[(i)] The maximum is $1 - \phi_{2}'(t,x_{0})$. In this case we have $1 - \phi_{1}'(t,y_{0})-(1 - \phi_{2}'(t,x_{0}))\leq 0$, which implies 
    $2\beta(x_{0}+y_{0})\leq 0$ and we come back to the case $x_{0}=y_{0}=0$, deducing
    \begin{equation}
        v_{2}(t,x)\leq v_{1}(t,x), \ \ \ \forall t\geq 0,\ \ \ \forall x\in (0,\infty)\,.
    \end{equation}
   \item[(ii)] The maximum is $H(t,x_{0},v_{2},\phi_{2}'(t,x_{0}))$. In this case we have\\ $$H(t,y_{0},v_{1},\phi_{1}'(t,y_{0}))-H(t,x_{0},v_{2},\phi_{2}'(t,x_{0}))\leq 0,$$
which implies 
\begin{align*}
    &\sup_{\alpha\in [\underline{\alpha}, 1]}\left[c(\alpha)\phi_{2}'(t,x_{0}) + \lambda\int_{\mathbb{R}_+} v_{2}(t,x_{0}- \alpha u)\, \mathrm{d} F(u) - (q+\lambda)v_{2}(t,x_{0})\right.\\
    &\qquad -\left.\left(c(\alpha)\phi_{1}'(t,y_{0}) + \lambda\int_{\mathbb{R}_+} v_{1}(t,y_{0}- \alpha u)\, \mathrm{d} F(u) - (q+\lambda)v_{1}(t,y_{0})\right)\right]\\
    &\geq \sup_{\alpha\in [\underline{\alpha}, 1]}\left[c(\alpha)\phi_{2}'(t,x_{0}) + \lambda\int_{\mathbb{R}_+} v_{2}(t,x_{0}- \alpha u)\, \mathrm{d}F(u) - (q+\lambda)v_{2}(t,x_{0})\right]\\
    &\qquad -\sup_{\alpha\in [\underline{\alpha}, 1]}\left[c(\alpha)\phi_{1}'(t,y_{0}) + \lambda\int_{\mathbb{R}_+} v_{1}(t,y_{0}- \alpha u)\, \mathrm{d}F(u) - (q+\lambda)v_{1}(t,y_{0})\right]\\
    &\geq 0\,.
\end{align*}
From the latter we deduce
\begin{align*}
 &   q(v_{2}(t,x_{0}) - v_{1}(t,y_{0})) \leq\sup_{\alpha\in [\underline{\alpha}, 1]}\left[c(\alpha)(\phi_{2}'(t,x_{0})-\phi_{1}'(t,y_{0}))\right. \\
  &  +\left. \lambda\int_{\mathbb{R}_+} (v_{2}(t,x_{0}- \alpha u)-v_{1}(t,y_{0}- \alpha u)-v_{2}(t,x_{0})+v_{1}(t,y_{0}))\, \mathrm{d}F(u)\right].
\end{align*}
Using the fact that $(x,y) \rightarrow \Psi(t, x, y)$ takes its maximum at $(x_0,y_0)$, we infer 
\begin{equation*}
    v_{2}(t,x_{0}) - v_{1}(t,y_{0})\leq\frac{1}{q}\sup_{\alpha\in [\underline{\alpha}, 1]}\left[2\beta c(\alpha)(x_{0}+y_{0}) + 2\lambda \beta\alpha^{2}\mathbb{E}(U^{2})-2\lambda\beta\alpha(x_{0}+y_{0}) \mathbb{E}(U) \right].
\end{equation*}
Recalling the expression of $c(\alpha)$ in \eqref{eq:c-alpha}, it holds
\begin{equation*}
    v_{2}(t,x_{0}) - v_{1}(t,y_{0})\leq\frac{2\beta \lambda}{q}\left[ \mathbb{E}(U)(1+\eta_{1})(x_{0}+y_{0}) + \mathbb{E}(U^{2}) \right].
\end{equation*}
Moreover, observing that $\Psi(t,x,x)\leq \Psi(t,x_{0},y_{0})$, for all $t\geq 0$, it follows
\begin{align*}
   & v_{2}(t,x)-v_{1}(t,x)-2\beta|x|^{2}\\
    &\quad \leq\frac{2\beta \lambda}{q}\left[ \mathbb{E}(U)(1+\eta_{1})(x_{0}+y_{0}) + \mathbb{E}(U^{2}) \right]- \beta(|x_{0}|^{2}+|y_{0}|^{2})\\
    &\quad \leq \beta|x_{0}|\left(\frac{2 \lambda\mathbb{E}(U)}{q}(1+\eta_{1})-|x_{0}|\right)+\beta|y_{0}|\left(\frac{2 \lambda\mathbb{E}(U)}{q}(1+\eta_{1})-|y_{0}|\right)+\frac{2\beta \lambda}{q}\mathbb{E}(U^{2}).
\end{align*}
We distinguish three cases:
\begin{enumerate}
    \item $x_{0}>\frac{2 \lambda\mathbb{E}(U)}{q}(1+\eta_{1})$ and $y_{0}>\frac{2 \lambda\mathbb{E}(U)}{q}(1+\eta_{1})$,
    \begin{equation*}
    v_{2}(t,x)-v_{1}(t,x)-2\beta|x|^{2}\leq \frac{2\beta \lambda}{q}\mathbb{E}(U^{2}).
\end{equation*}
    Passing $\beta \downarrow 0$, one gets
    \begin{equation*}
    v_{2}(t,x)\leq v_{1}(t,x), \ \ \ \forall t\geq 0,\ \ \ \forall x\geq 0.
    \end{equation*} 
    \item $x_{0}\leq \frac{2 \lambda\mathbb{E}(U)}{q}(1+\eta_{1})$ and $y_{0}\leq \frac{2 \lambda\mathbb{E}(U)}{q}(1+\eta_{1})$,
    \begin{equation*}
    v_{2}(t,x)-v_{1}(t,x)-2\beta|x|^{2}\leq 2\beta\left(\frac{2 \lambda\mathbb{E}(U)}{q}(1+\eta_{1})\right)^{2}+\frac{2\beta \lambda}{q}\mathbb{E}(U^{2}).
\end{equation*}
    Passing $\beta \downarrow 0$, one gets
    \begin{equation*}
    v_{2}(t,x)\leq v_{1}(t,x), \ \ \ \forall t\geq 0,\ \ \ \forall x\geq 0.
    \end{equation*}
    \item $x_{0}>\frac{2 \lambda\mathbb{E}(U)}{q}(1+\eta_{1})$ and $y_{0}\leq\frac{2 \lambda\mathbb{E}(U)}{q}(1+\eta_{1})$,
    \begin{equation*}
    v_{2}(t,x)-v_{1}(t,x)-2\beta|x|^{2}\leq \beta\left(\frac{2 \lambda\mathbb{E}(U)}{q}(1+\eta_{1})\right)^{2}+\frac{2\beta \lambda}{q}\mathbb{E}(U^{2}).
    \end{equation*}
    Passing $\beta \downarrow 0$, one gets
    \begin{equation*}
    v_{2}(t,x)\leq v_{1}(t,x), \ \ \ \forall t\geq 0,\ \ \ \forall x\geq 0
    \end{equation*}
\end{enumerate}
and the statement of the proposition follows.\\

\end{enumerate}
\end{proof}
\noindent The following theorem shows the comparison principle on $(-a,0)$.
\begin{theorem}\label{THC2} 
We fix $t\geq 0$. Let $v_{1}(t, \cdot)$ and $v_{2}(t, \cdot)$ be respectively a viscosity super-solution and a viscosity sub-solution of (\ref{HJB2}) on $(-a,0)$. We assume that $v_1$ and $v_2$ are uniformly continuous and  
\begin{eqnarray}\label{civ1v22}
v_{1}(t,0)\geq v_{2}(t,0)\mbox{ and }v_{1}(t,-a)\geq v_{2}(t,-a),
\end{eqnarray}
Then it holds 
\begin{equation*}
    v_{1}(t,x)\geq v_{2}(t,x), \forall x\in (-a, 0)\,.
\end{equation*}
\end{theorem}

\section{Explicit expressions for optimal policies, for the Cram\'er-Lundberg process with exponential jumps} \label{sec:Guess} 
In this section we derive through the {\it guess-and-verify} method, an explicit expression for the value function associated to $(\alpha, a, b)$ 
policies consisting in proportional reinsurance, injecting capital up till the level $0$, provided that the severity of ruin does not reach $a \geq 0$ and paying dividends as soon as the process reaches some upper level $b \geq 0$. To derive our results, we adapt the computations presented in \cite[Section 3]{avram2021equity} to our setting. Notice that in our case $\alpha$ is piece-wise constant and the aim is to calculate, for $i=1,2,\ldots$, $\alpha_t=\alpha_{t_{i}}$ at each time step $t\in [t_i, t_{i+1})$. 
\\
Let $X_s^{t,x,\alpha}$, $t\geq s\geq  0$ be the Cramér-Lundberg process as defined in (\ref{Xalpha}). Since the jumps in this model are non-positive, we know that the moment generating function is well defined for $\theta \in \R_+$ as
\begin{equation}\label{k}
    \kappa(t_{i},\theta) := \log \mathbb{E}_{x_{i}}[\mathrm{e}^{\theta (X_{(t_{i+1})^-}^{t_i,x_i,\alpha}-X_{t_{i}}^{t_i,x_i,\alpha})}]= (t_{i+1}-t_i)\left[c_{\alpha_{t_{i}}}\theta +\int_{0}^{\infty}(\mathrm{e}^{-\alpha_{t_{i}}\theta z}-1) \lambda \,F(\mathrm{d}z)\right]\,,
\end{equation}
where $c_{\alpha_{t_{i}}}=c - (1-\alpha_{t_{i}}) (1+\eta_2)\lambda \mu\,.$ From now on we assume $t_{i+1}-t_i=1$, $\forall i \in \mathbb{N}$ and
exponential jump-sizes with mean $1/\mu$, jump rate $\lambda$ and premium rate $c_{\alpha_{t_{i}}} > 0$. The  Laplace
exponent in this case is given by 
$$ \kappa(t_{i},\theta) = \left(\theta c_{\alpha_{t_{i}}}-\frac{\alpha_{t_{i}}\lambda \theta}{\mu+\alpha_{t_{i}}\theta}\right)\,.$$ 
Solving $\kappa(t_{i},\theta) -q = 0$, for $\theta$ yields
$$\alpha_{t_{i}} c_{\alpha_{t_{i}}}\theta^{2}+(c_{\alpha_{t_{i}}}\mu -\alpha_{t_{i}}\lambda-\alpha_{t_{i}}q)\theta-\mu q=0,$$
and we get two distinct solutions $\rho_{-}\leq 0\leq \rho_{+}=\Phi_{q}$ given by
\begin{equation}\label{eq:rho}
\rho_{\pm}=\frac{1}{2 \alpha_{t_{i}}c_{\alpha_{t_{i}}}}\left(-(c_{\alpha_{t_{i}}}\mu -\alpha_{t_{i}}\lambda-\alpha_{t_{i}}q)\pm\sqrt{(c_{\alpha_{t_{i}}}\mu -\alpha_{t_{i}}\lambda-\alpha_{t_{i}}q)^{2}+4\alpha_{t_{i}}c_{\alpha_{t_{i}}}\mu q}\right)\,.
\end{equation}
We gather first some results in the following proposition that are useful to derive the explicit expression of the value function. 
\begin{Proposition}\label{Prop4.1}
Let $\mathcal{L}^{-1}$ be the inverse Laplace transform of $\theta \rightarrow (\kappa(t_{i},\theta)-q)^{-1}$, where $\kappa(t_{i},\theta)$ is  as defined in \eqref{k}. For $ t \in [t_{i},t_{i+1})$ we define $W_q:\R_+^2 \rightarrow \R$ and $Z_q: \R_+^2\rightarrow \R$ respectively as 
\begin{equation}\label{eq:Wq}
W_{q}(t,x):=\mathcal{L}^{-1}\left\{\frac{1}{ \kappa(t_{i},s) -q}\right\}=\frac{(\alpha_{t_{i}}\Phi_{q}+\mu)\mathrm{e}^{\Phi_{q}x}-(\alpha_{t_{i}}\rho_{-}+\mu)\mathrm{e}^{\rho_{-}x}}{\alpha_{t_{i}}c_{\alpha_{t_{i}}}(\Phi_{q}-\rho_{-})}
\end{equation}
and 
$Z_{q}(t,x):=1+q\int_{0}^{x}W_{q}(t,y)\,\mathrm{d}y$. Moreover, Let $\mathcal{G}: C^1(\R_+^{2}) \rightarrow \R$ be defined as
    \begin{equation}\label{eq41}
        \mathcal{G}\phi(t_{i},y) =c_{\alpha_{t_{i}}}\phi'(t_{i},y)+\lambda \int_{\mathbb{R}_{+}} \left(\phi(t_{i},y-z)-\phi(t_{i},y)\right)\, \mathrm{d}F(\frac{z}{\alpha_{t_{i}}})-q\phi(t_{i},y)\,.
    \end{equation}
It holds 
\begin{enumerate}
    \item $\Phi_{q}+\rho_{-}=-\frac{c_{\alpha_{t_{i}}}\mu-\alpha_{t_{i}}\lambda-\alpha_{t_{i}}q}{\alpha_{t_{i}}c_{\alpha_{t_{i}}}};\qquad  \Phi_{q}\rho_{-}=-\frac{\mu q}{\alpha_{t_{i}}c_{\alpha_{t_{i}}}}.$
    \item $c_{\alpha_{t_{i}}}\rho(\alpha_{t_{i}}\rho+\mu)=\alpha_{t_{i}}(\lambda+q)\rho+\mu q, \qquad \rho \in \{\Phi_{q}, \rho_{-}\}$\,. \label{asser-2}
    \item \label{item:3} $c_{\alpha_{t_{i}}}W_{q}(t,x)-Z_{q}(t,x)=\frac{\lambda}{c_{\alpha_{t_{i}}}(\Phi_{q}-\rho_{-})}(\mathrm{e}^{\Phi_{q}x}-\mathrm{e}^{\rho_{-}x})\geq 0\,, \qquad \forall x\geq 0.$
    \item The functions $W_{q}$ and $Z_{q}$ (extended by $W_{q}(t,x)=0, Z_{q}(t,x) = 1, \forall x < 0$) satisfy the equation $\mathcal{L}\phi(t_{i},b) = 0$, for all $b \geq 0$, i.e.,
    \begin{equation}\label{eq:Wq-generator}
    c_{\alpha_{t_{i}}}W_{q}'(t,b)+\lambda \int_{0}^{b} W_{q}(t,y)\frac{\mu}{\alpha_{t_{i}}} \mathrm{e}^{-\frac{\mu}{\alpha_{t_{i}}}(b-y)}\,\mathrm{d}y-(\lambda+q)W_{q}(t,b)=0\,,
    \end{equation}
    \begin{equation}\label{eq:Zq-generator}
    c_{\alpha_{t_{i}}}Z_{q}'(t,b)+\lambda \int_{-\infty}^{b} Z_{q}(t,y)\frac{\mu}{\alpha_{t_{i}}} \mathrm{e}^{-\frac{\mu}{\alpha_{t_{i}}}(b-y)}\,\mathrm{d}y-(\lambda+q)Z_{q}(t,b)=0.
    \end{equation}
\end{enumerate}
\end{Proposition}
\begin{proof}
\begin{enumerate}
\item A mere consequence of the definition of $\Phi_{q}, \ \rho_{-}$. \label{item-1}
\item It holds \label{item-2}
\begin{align*}
&  c_{\alpha_{t_{i}}}\Phi_{q}(\alpha_{t_{i}}\Phi_{q}+\mu)\\
&\quad =\alpha_{t_{i}}c_{\alpha_{t_{i}}}\Phi_{q}^{2}+c_{\alpha_{t_{i}}}\mu\Phi_{q}\\
    &\quad = \frac{(c_{\alpha_{t_{i}}}\mu -\alpha_{t_{i}}\lambda-\alpha_{t_{i}}q)}{2 \alpha_{t_{i}}c_{\alpha_{t_{i}}}}\left(c_{\alpha_{t_{i}}}\mu -\alpha_{t_{i}}\lambda-\alpha_{t_{i}}q -\sqrt{(c_{\alpha_{t_{i}}}\mu -\alpha_{t_{i}}\lambda-\alpha_{t_{i}}q)^{2}+4\alpha_{t_{i}}c_{\alpha_{t_{i}}}\mu q}\right)\\
    &\qquad \quad +\mu q+c_{\alpha_{t_{i}}}\mu\Phi_{q} \\
    &\quad =\alpha_{t_{i}}(\lambda+q)\Phi_{q}-c_{\alpha_{t_{i}}}\mu\Phi_{q}+c_{\alpha_{t_{i}}}\mu\Phi_{q}+\mu q 
   \\
    &\quad =\alpha_{t_i}^{*}(\lambda+q)\Phi_{q}+\mu q.  
\end{align*}
We deduce the equality for $\rho_-$ similarly.
\item We deduce from assertions \ref{item-1} and \ref{item-2},
\begin{align*}
    &c_{\alpha_{t_{i}}}W_{q}(t,x)-Z_{q}(t,x)=\frac{1}{c_{\alpha_{t_{i}}}(\Phi_{q}-\rho_{-})}\left(c_{\alpha_{t_{i}}}(\alpha_{t_{i}}\Phi_{q}+\mu)\mathrm{e}^{\Phi_{q}x}-c_{\alpha_{t_{i}}}(\alpha_{t_{i}}\rho_{-}+\mu)\mathrm{e}^{\rho_{-}x}\right.\\
& \left. -q\frac{(\alpha_{t_{i}}\Phi_{q}+\mu)(\mathrm{e}^{\Phi_{q}x}-1)}{\Phi_{q}}+q\frac{(\alpha_{t_{i}}\rho_{-}+\mu)(\mathrm{e}^{\rho_{-}x}-1)}{\rho_{-}} \right)-1\\
    &=\frac{\lambda}{c_{\alpha_{t_{i}}}(\Phi_{q}-\rho_{-})}(\mathrm{e}^{\Phi_{q}x}-\mathrm{e}^{\rho_{-}x}).
 \end{align*}
\item  We write $W_{q}(t,x)=\frac{\alpha_{t_{i}}\Phi_{q}+\mu}{\alpha_{t_{i}}c_{\alpha_{t_{i}}}(\Phi_{q}-\rho_{-})}\mathrm{e}^{\Phi_{q}x}-\frac{\alpha_{t_{i}}\rho_{-}+\mu}{\alpha_{t_{i}}c_{\alpha_{t_{i}}}(\Phi_{q}-\rho_{-})}\mathrm{e}^{\rho_{-}x}$. Then we compute, using the previous assertions,
\begin{align*}
c_{\alpha_{t_{i}}}W_{q}'(t,x)&=\frac{c_{\alpha_{t_{i}}}\Phi_{q}(\alpha_{t_{i}}\Phi_{q}+\mu)}{\alpha_{t_{i}}c_{\alpha_{t_{i}}}(\Phi_{q}-\rho_{-})}\mathrm{e}^{\Phi_{q}x}-\frac{c_{\alpha_{t_{i}}}\rho_{-}(\alpha_{t_{i}}\rho_{-}+\mu)}{\alpha_{t_{i}}c_{\alpha_{t_{i}}}(\Phi_{q}-\rho_{-})}\mathrm{e}^{\rho_{-}x}\\
    &=\frac{\alpha_{t_{i}}(\lambda+q)\Phi_{q}+\mu q}{\alpha_{t_{i}}c_{\alpha_{t_{i}}}(\Phi_{q}-\rho_{-})}\mathrm{e}^{\Phi_{q}x}-\frac{\alpha_{t_{i}}(\lambda+q)\rho_{-}+\mu q}{\alpha_{t_{i}}c_{\alpha_{t_{i}}}(\Phi_{q}-\rho_{-})}\mathrm{e}^{\rho_{-}x},
\end{align*}
\begin{align*}
\lambda \int_{0}^{\infty} W_{q}(t,x-\alpha_{t_{i}}y)\mu \mathrm{e}^{-\mu y}\, \mathrm{d}y&=\lambda\mu \mathrm{e}^{-\frac{\mu x}{\alpha_{t_{i}}}} \int_{0}^{\frac{x}{\alpha_{t_{i}}}} W_{q}(t,\alpha_{t_{i}}y) \mathrm{e}^{\mu y}\,\mathrm{d}y\\
    &=\frac{\lambda\mu}{\alpha_{t_{i}}c_{\alpha_{t_{i}}}(\Phi_{q}-\rho_{-})}(\mathrm{e}^{\Phi_{q}x}-\mathrm{e}^{\rho_{-}x}),
\end{align*}
and 
\begin{align*}
  -(\lambda+q)W_{q}(t,x)&=-(\lambda+q)\frac{\alpha_{t_{i}}\Phi_{q}+\mu}{\alpha_{t_{i}}c_{\alpha_{t_{i}}}(\Phi_{q}-\rho_{-})}\mathrm{e}^{\Phi_{q}x}+(\lambda+q)\frac{\alpha_{t_{i}}\rho_{-}+\mu}{\alpha_{t_{i}}c_{\alpha_{t_{i}}}(\Phi_{q}-\rho_{-})}\mathrm{e}^{\rho_{-}x}\\
    &=\frac{-(\lambda+q)\alpha_{t_{i}}\Phi_{q}}{\alpha_{t_{i}}c_{\alpha_{t_{i}}}(\Phi_{q}-\rho_{-})}\mathrm{e}^{\Phi_{q}x}+\frac{(\lambda+q)\alpha_{t_{i}}\rho_{-}}{\alpha_{t_{i}}c_{\alpha_{t_{i}}}(\Phi_{q}-\rho_{-})}\mathrm{e}^{\rho_{-}x}\\
    &\qquad +\frac{-(\lambda+q)\mu}{\alpha_{t_{i}}c_{\alpha_{t_{i}}}(\Phi_{q}-\rho_{-})}(\mathrm{e}^{\Phi_{q}x}-\mathrm{e}^{\rho_{-}x}),
\end{align*}
which implies equation \ref{eq:Wq-generator} for $W_q$. Equation \ref{eq:Zq-generator} for $Z_{q}$ follows similarly.
\end{enumerate}
\end{proof}
\subsection{Explicit-expression for the optimal value function}
\noindent In the following proposition, we derive an explicit expression for the value function in terms of functions $S, G:\R_+\times \R_+ \rightarrow \R$ that we define for $t \in [t_i, t_{i+1})$ and $s\in \R_+$ as follows 
\begin{equation}
\left \{
\begin{array}{rcl}
S(s,x)&=& \mathrm{e}^{-\frac{\mu}{\alpha_{t_{i}}} s}C_{q}(x)+Z_{q}(t,x)\,, \\
G(s,x) &=& m_{\alpha_{t_{i}}}(s)C_{q}(x)\,,
\end{array}
\right.
\end{equation}
where
\begin{equation}\label{eq:Cq}
    C_{q}(x)=c_{\alpha_{t_{i}}}W_{q}(t,x)-Z_{q}(t,x)\,, \ \ \forall \ x\geq 0
\end{equation}
and 
\begin{equation}\label{eq:mr}
    m_{r}(s)=\int_{0}^{s}y\mu_{r} \mathrm{e}^{-\mu_{r} y}\mathrm{d}y=\frac{1-\mathrm{e}^{-\mu_{r} s}(\mu_{r} s +1)}{\mu_{r}},   \ \ \ \mu_{r}=\frac{\mu}{r}, \ \ \ \ r> 0
\end{equation}
denotes the mean function of our claims cut at level $s\geq 0$. When there is no ambiguity, we will use $m(s)=m_{\alpha_{0}}(s)$. Notice that $S(s, 0) = 1$, $C_{q}(0)=G(s, 0)=0$, $\forall s \geq 0$, $C_{q}(0+)=\lambda/c_{\alpha_{t_{i}}}$ and that $x \rightarrow C_{q}(x)$, $x\rightarrow G(s, x)$, $x\rightarrow S(s, x)$ are non-decreasing functions, for all $s \geq 0$.\\
Recall the value function in \eqref{eq:function-J}. For ease of notation, we use in this subsection
\begin{equation}\label{eq:notation-J}
J_{t_i,x}:= J(t_i,x,\alpha, a,b)
\end{equation}
to denote the expected discounted dividends and capital injections associated to policies $\alpha, a,b$ consisting
in capital injections with proportionality cost $k>1$, provided that the severity of ruin is smaller than $a >0$ (and declaring bankruptcy for larger severity), paying dividends as soon as the
process reaches some upper level b, and the proportional reinsurance the company picks. In the following proposition we derive an expression for the value function in terms of the function $S$ and $G$ introduced above. 
\begin{Proposition}\label{Prop4.2}
Let $X_s^{t,x,\alpha}$, $t\in [t_i, t_{i+1})$, $t\geq s\geq  0$ be the Cramér-Lundberg process as defined in (\ref{Xalpha}) with jump-sizes being exponentially distributed with mean $1/\mu$. Let $t_{i+1}-t_i=1$, $\forall i \in \mathbb{N}$ and $J$ be as in \eqref{eq:notation-J}. 
We set
\begin{equation}\label{zaj}
    z(a,J_{t_{i},0}) :=\int_{0}^{a}(J_{t_{i},0}-ky)\frac{\mu}{\alpha_{t_{i}}} \mathrm{e}^{-\frac{\mu}{\alpha_{t_{i}}} y}\mathrm{d}y = J_{t_{i},0}(1-\mathrm{e}^{-\frac{\mu}{\alpha_{t_{i}}} a}) - km_{\alpha_{t_{i}}}(a).
\end{equation}
It holds
\begin{enumerate}
    \item The cost function satisfies, for $0 \leq x \leq b$,\\
    \begin{equation}\label{eq:JX}
        J_{t_{i},x}=\frac{W_{q}(t,x)}{W_{q}(t,b)}J_{t_{i},b}+\left(Z_{q}(t,x)-\frac{W_{q}(t,x)}{W_{q}(t,b)}Z_{q}(t,b)\right)z(a,J_{t_{i},0})\,.
    \end{equation}
    In particular, if we set
    \begin{equation}\label{jx}
        j_{x}=J_{t_{i},x}-z(a,J_{t_{i},0})Z_{q}(t,x),
    \end{equation}
    then
    \begin{equation}\label{jxx}
        j_{x}=\frac{W_{q}(t,x)}{W_{q}(t,b)}j_{b}, \ \ \ 0 \leq x \leq b\,.
    \end{equation}
    \item The cost can be explicitly written as
    \begin{equation}
        J_{t_{i},x}=\left\{ \begin{array}{ll}
    0\,, & \mbox{if}\,\,  x< -a\,, \\
        kx+J_{t_{i},0}& \mbox{if}\,\, x\in [-a, 0]\,, \\
        kG(a, x)+J_{t_{i},0}S(a, x)=kG(a, x)+\frac{1-k\partial_{b}G(a, b)}{\partial_{b}S(a, b)}S(a, x)\,, \ \ &\mbox{if}\,\,  x\in (0, b]
        \end{array} \right.
    \end{equation}
    and
    \begin{equation}\label{J0}
        J_{t_{i},0}=\frac{1-k\partial_{b}G(a, b)}{\partial_{b}S(a, b)}=\frac{1-k[m_{\alpha_{t_{i}}}(a)C'_{q}(b)]}{\mathrm{e}^{-\frac{\mu}{\alpha_{t_{i}}} a}C'_{q}(b)+qW_{q}(t,b)}\,.
    \end{equation}
\end{enumerate}
\end{Proposition}     
\begin{proof}
\begin{enumerate}
\item Let $\tau^{x}=\tau^{x}_{0-}\wedge\tau^{x}_{b+}=\inf\{s\geq t_{i}: \, X_{s}^{t_{i},x,\alpha,\pi}<0\}\wedge\inf\{s\geq t_{i}: \, X_{s}^{t_{i},x,\alpha,\pi}>b\}$. By applying the strong Markov property at the stopping time and ~\cite[Theorem 8.1]{kyprianou2014fluctuations}, we get the following statement, for $0\leq x\leq b$,
\begin{align}\label{Jxx}
         J_{t_{i},x}=&\mathbb{E}_{x}\left[\mathrm{e}^{-q(\tau^{x}_{b+}-t_{i})}\mathbb{1}_{\tau^{x}_{b+}<\tau^{x}_{0-}}\right]J_{t_{i},b}+\mathbb{E}_{x}\left[\mathrm{e}^{-q(\tau^{x}_{0-}-t_{i})}\mathbb{1}_{\tau^{x}_{b+}>\tau^{x}_{0-}}\left(J_{t_{i},0}+kX^{t_{i},x,\alpha}_{\tau^{x}_{0-}} \right)\mathbb{1}_{X^{t_{i},x,\alpha}_{\tau^{x}_{0-}}\geq-a}\right]\nonumber\\
            =&\frac{W_{q}(t,x)}{W_{q}(t,b)}(J_{t_{i},b}-I_{b})+I_{x}\,,
\end{align}    
where $I_{y}=\mathbb{E}_{y}\left[\mathrm{e}^{-q(\tau_{0-}^y-t_{i})}\left(J_{t_{i},0}+kX^{t_{i},y,\alpha}_{\tau_{0-}} \right)\mathbb{1}_{X^{t_{i},y,\alpha}_{\tau_{0-}}\geq-a}\right]$ and the random times are understood to be applied
to the controlled reserve starting at $y$. The term $I_{y}$ can be explicitly computed, using the Gerber-Shiu measure, see \cite{kyprianou2013gerber} for more details, 
\begin{align*}
    I_{y}=&\lambda \int_{0}^{a}\frac{\mu}{\alpha_{t_{i}}} \mathrm{e}^{-\frac{\mu}{\alpha_{t_{i}}} u}(J_{t_{i},0}+ku)\,\mathrm{d}u\times \int_{\mathbb{R}_{+}} \left(\mathrm{e}^{-\frac{(\alpha_{t_{i}}\Phi_{q}+\mu)}{\alpha_{t_{i}}}v}W_{q}(t,y)-\mathrm{e}^{-\frac{\mu}{\alpha_{t_{i}}}v}W_{q}(t,y-v)\right)\, \mathrm{d}v\\
    =&\lambda \left\{J_{t_{i},0}(1-\mathrm{e}^{-\frac{\mu}{\alpha_{t_{i}}} a}) - km_{\alpha_{t_{i}}}(a) \right\} \left(\frac{\alpha_{t_{i}} W_{q}(t,y)}{\alpha_{t_{i}}\Phi_{q}+\mu}- \int_{0}^{y} \mathrm{e}^{-\frac{\mu}{\alpha_{t_{i}}}v}W_{q}(t,y-v)\,\mathrm{d}v\right)\,. 
\end{align*}
It follows, using the expression for $W_q$ in \eqref{eq:Wq} and Proposition \ref{Prop4.1}, \ref{item:3}.  
$$I_{y}=\left\{J_{t_{i},0}(1-\mathrm{e}^{-\frac{\mu}{\alpha_{t_{i}}} a}) - km_{\alpha_{t_{i}}}(a) \right\} \left(Z_{q}(t,y)-\frac{q}{\Phi_{q}}W_{q}(t,y)\right)\,.$$
Equation \eqref{eq:JX} follows by plugging the expression of $I_{y}$ into (\ref{Jxx}) and (\ref{jxx}) is a mere consequence of \eqref{eq:JX}.
\item Putting $x = 0$ in (\ref{jx}) yields $J_{t_{i},0}=J_{t_{i},0}(1-\mathrm{e}^{-\frac{\mu}{\alpha_{t_{i}}} a}) - km_{\alpha_{t_{i}}}(a)+\frac{j_{b}}{W_{q}(t,b)}W_{q}(t,0)$. Hence, it holds 
\begin{equation}\label{jb}
    \frac{j_{b}}{W_{q}(t,b)} W_{q}(t,0) =km_{\alpha_{t_{i}}}(a)+J_{t_{i},0} \mathrm{e}^{-\frac{\mu}{\alpha_{t_{i}}} a}=: w(a,J_{t_{i},0}).
\end{equation}
Using \eqref{jb}, the fact that $W_{q}(t,0) =\frac{1}{c_{\alpha_{t_{i}}}}\ne0$, together with (\ref{jxx}), yield for all $x \in [0, b]$
\begin{align}
    J_{t_{i},x} &=z(a,J_{t_{i},0})Z_{q}(t,x)+j_{x}\nonumber\\
            &=z(a,J_{t_{i},0})Z_{q}(t,x)+c_{\alpha_{t_{i}}}w(a,J_{t_{i},0})W_{q}(t,x)\label{Jx146}\\
            &=\left(J_{t_{i},0}(1-\mathrm{e}^{-\frac{\mu}{\alpha_{t_{i}}} a}) - km_{\alpha_{t_{i}}}(a) \right)Z_{q}(t,x)+c_{\alpha_{t_{i}}}\left(J_{t_{i},0} \mathrm{e}^{-\frac{\mu}{\alpha_{t_{i}}} a}+km_{\alpha_{t_{i}}}(a)\right)W_{q}(t,x)\nonumber\\
            &=J_{t_{i},0}S(a, x)+kG(a, x). \label{Jx46}
\end{align}
We use another equation for $J_{t_{i},b}$, obtained by conditioning at the time of the first claim when starting from $b$, we get
\begin{align}\label{eq:Jtb}
J_{t_{i},b}&=\mathbb{E}_{b}\left[\int_{t_{i}}^{\tau_{1}}\mathrm{e}^{-q(t-t_{i})}c_{\alpha_{t_{i}}}dt+\mathrm{e}^{-q(\tau_{1}-t_{i})}\int_{0}^{b+a}J_{t_{i},b-y}\frac{\mu}{\alpha_{t_{i}}} \mathrm{e}^{-\frac{\mu}{\alpha_{t_{i}}} y}\mathrm{d}y\right]\nonumber\\
    &=\frac{c_{\alpha_{t_{i}}}}{\lambda+q}+\frac{\lambda}{\lambda+q}\left(\int_{b}^{b+a}(J_{t_{i},0}+k(b-z))\frac{\mu}{\alpha_{t_{i}}} \mathrm{e}^{-\frac{\mu}{\alpha_{t_{i}}} z}\mathrm{d}z\right.\nonumber\\
    &\qquad +\left. \int_{0}^{b}[z(a,J_{t_{i},0})Z_{q}(t,y)+c_{\alpha_{t_{i}}}w(a,J_{t_{i},0})W_{q}(t,y)]\frac{\mu}{\alpha_{t_{i}}} \mathrm{e}^{-\frac{\mu}{\alpha_{t_{i}}} (b-y)}\mathrm{d}y \right)\nonumber\\
    &=\frac{c_{\alpha_{t_{i}}}}{\lambda+q}+\frac{\lambda}{\lambda+q}\left[J_{t_{i},0}\left(\mathrm{e}^{-\frac{\mu}{\alpha_{t_{i}}} b}-\mathrm{e}^{-\frac{\mu}{\alpha_{t_{i}}} (a+b)}\right)-k\mathrm{e}^{-\frac{\mu}{\alpha_{t_{i}}} b}m_{\alpha_{t_{i}}}(a)\right]\nonumber\\
    &\qquad + \frac{1}{\lambda+q}\left\{z(a,J_{t_{i},0})\left((\lambda+q)Z_{q}(t,b)-c_{\alpha_{t_{i}}}Z_{q}'(t,b_{+})-\lambda \mathrm{e}^{-\frac{\mu}{\alpha_{t_{i}}}b} \right)\right.\nonumber\\
    &\qquad \left.+c_{\alpha_{t_{i}}}w(a,J_{t_{i},0})\left((\lambda+q)W_{q}(t,b)-c_{\alpha_{t_{i}}}W_{q}'(t,b_{+})\right) \right\}\,,
\end{align}
where the second equality follows by observing that $\tau-t_i$ is exponentially distributed with mean $1/\lambda$ and by exploiting \eqref{Jx146} and the third equality follows using the last assertion in Proposition \ref{Prop4.1}. As a consequence, expressing equation \eqref{eq:Jtb} in terms of the function $z$ in \eqref{zaj}, then
recalling that $J_{t_{i},b}=z(a,J_{t_{i},0})Z_{q}(t,b)+c_{\alpha_{t_{i}}}w(a,J_{t_{i},0})W_{q}(t,b)$ and using again the expression for $z$ and for $w$ in \eqref{jb}, we get
\begin{align*}
    &J_{t_{i},0}c_{\alpha_{t_{i}}}\left((1-\mathrm{e}^{-\frac{\mu}{\alpha_{t_{i}}}a})Z_{q}'(t,b_{+})+c_{\alpha_{t_{i}}}\mathrm{e}^{-\frac{\mu}{\alpha_{t_{i}}}a}W_{q}'(t,b_{+})\right)\\
&=c_{\alpha_{t_{i}}}+c_{\alpha_{t_{i}}}km_{\alpha_{t_{i}}}(a)Z_{q}'(t,b_{+})-c_{\alpha_{t_{i}}}^{2}km_{\alpha_{t_{i}}}(a) W_{q}'(t,b_{+})\,,
\end{align*}
or, again, 
\begin{align}\label{eq:expression-dJ}
    c_{\alpha_{t_{i}}}J_{t_{i},0}\partial_{b}S(a, b) &=c_{\alpha_{t_{i}}}[1-km_{\alpha_{t_{i}}}(a)(c_{\alpha_{t_{i}}} W_{q}'(t,b_{+})-Z_{q}'(t,b_{+})]\nonumber\\
    &=c_{\alpha_{t_{i}}}[1-k\partial_{b}G(a, b)]\,.
\end{align}
Equality (\ref{J0}) follows. By plugging \eqref{eq:expression-dJ} into (\ref{Jx46}), we derive the expression \eqref{jxx} and the statements of the proposition are proved.
\end{enumerate}
\end{proof}
\subsection{Derivations concerning the optimal parameters $\alpha^{*}$, $a^{*}$ and $b^{*}$} 
Recall $\phi_q$ and $ \rho$ in \eqref{eq:rho}. In the previous section, we found it convenient to express the cost $J_{t_i,x}$ in terms of the
functions $S$ and $G$. In this section we introduce the following functions $\gamma, \theta: \R_+\times \R \rightarrow \R$ that we use to express the cost $J_{t_i,x}$. This will be convenient for our analysis concerning the optimal parameters $a^*$, $b^*$, and $\alpha^*$.
\begin{equation}
    \gamma(x,\alpha_{t_{i}})=\frac{\Phi_{q}(\alpha_{t_{i}})-\rho_{-}(\alpha_{t_{i}})}{\Phi_{q}(\alpha_{t_{i}})\mathrm{e}^{\Phi_{q}(\alpha_{t_{i}})x}-\rho_{-}(\alpha_{t_{i}})\mathrm{e}^{\rho_{-}(\alpha_{t_{i}})x}}
\end{equation}
and
\begin{equation}
\theta(x,\alpha_{t_{i}})=\frac{\mathrm{e}^{\Phi_{q}(\alpha_{t_{i}})x}-\mathrm{e}^{\rho_{-}(\alpha_{t_{i}})x}}{\Phi_{q}(\alpha_{t_{i}})\mathrm{e}^{\Phi_{q}(\alpha_{t_{i}})x}-\rho_{-}(\alpha_{t_{i}})\mathrm{e}^{\rho_{-}(\alpha_{t_{i}})x}}\,.
\end{equation}
 Using \eqref{J0} and \eqref{eq:Cq}, we express $J$ in terms of the functions $\gamma$ and $\theta$ as follows 
\begin{equation}\label{Jabf}
J(t_i,0,\alpha_{t_{i}}, a,b)=\frac{c_{\alpha_{t_{i}}}\gamma(b,\alpha_{t_{i}})-\lambda km_{\alpha_{t_{i}}}(a)}{q+\frac{\mu q}{\alpha_{t_{i}}} \theta(b,\alpha_{t_{i}})+\lambda \mathrm{e}^{-\frac{\mu}{\alpha_{t_{i}}} a}}.
\end{equation}
Since $\alpha_{t_{i}}$ is fixed in $[\underline{\alpha}, 1]$, for simplicity in the notation, from now on, we use $\alpha$ instead of $\alpha_{t_{i}}$.

\noindent We note first that $J(t_i,0,\alpha, a,\infty)\leq 0,$ and, thus, $b=\infty$ can never be optimal. For this purpose, one notes that $\lim_{x\rightarrow{\infty}}\gamma(x,\alpha)=0$. It follows that either $b^{*} = 0$ or it is a critical point of $(0,\infty)$.Moreover, we note that $J(t_i,0,\underline{\alpha}, a,b)\leq 0,$ and, thus, $\alpha=\underline{\alpha}$ can never be optimal. It follows that either $\alpha^{*} = 1$ or it is a critical point of $(\underline{\alpha},1)$.\\
Before we state our main conclusions, we start with a preparatory result, obtained by differentiating (\ref{J0}) with respect to $a$.
\begin{lemma}\label{akb}
 The partial derivative of $J(t_i,0,\alpha, a,b)$ with respect to $a$ (on $(0,\infty)$) is given by
\begin{equation}
    \frac{\partial }{\partial a}J(t_i,0,\alpha, a,b)=\lambda \frac{\mu}{\alpha}  \mathrm{e}^{-\frac{\mu}{\alpha} a} \frac{-ak(q+q\frac{\mu}{\alpha}\theta(b,\alpha))+c(\alpha)\gamma(b,\alpha)+\lambda k \frac{\mathrm{e}^{-\frac{\mu}{\alpha} a}-1}{\frac{\mu}{\alpha}}}{(q+\frac{\mu q}{\alpha} \theta(b,\alpha)+\lambda \mathrm{e}^{-\frac{\mu}{\alpha} a})^{2}}\,.
\end{equation}
As a consequence,
\begin{enumerate}
    \item picking $a \in \{0,\infty\}$ can never be optimal.
    \item for fixed $k>1$, $\alpha\in [\underline{\alpha},1],$ and $b \geq 0$, there exists a unique critical point $a_{k,b,\alpha}\ne 0$ satisfying $\frac{\partial }{\partial a}J(t_i,0,\alpha, a,b)=0$,
which is equivalent to
\begin{equation}\label{da0}
    -a_{k,b,\alpha}k(q+q\frac{\mu}{\alpha}\theta(b,\alpha))+c(\alpha)\gamma(b,\alpha)+\lambda k \frac{\mathrm{e}^{-\frac{\mu}{\alpha} a_{k,b,\alpha}}-1}{\frac{\mu}{\alpha}}=0.
\end{equation}
Furthermore, (\ref{da0}) implies that at $a_{k,b,\alpha}$, our objective simplifies 
\begin{equation}\label{jofa}
    J(t_i,0,\alpha, a_{k,b,\alpha},b)=ka_{k,b,\alpha}.
\end{equation}
\end{enumerate}
\end{lemma}
\noindent Thus, the optimal policy can neither be of Shreve-Lehoczky-Gaver type \cite{shreve1984optimal} (i.e. systematic injection to keep the reserve positive independently of the severity of the ruin), nor of De Finetti type \cite{de1957impostazione}. 
\begin{proof}
\begin{enumerate}
    \item To see that $a = 0$ cannot provide the maximal value for $J^{a,b,\alpha}_{0}$, one notes that 
    $$\lim_{a \rightarrow 0+} \frac{\partial }{\partial a}J(t_i,0,\alpha, a,b)=\frac{\lambda \mu c(\alpha)\gamma(b,\alpha)}{\alpha(q+\frac{\mu q}{\alpha} \theta(b,\alpha)+\lambda)^{2}}>0\,, \ \ \  \forall \, b\geq 0,\ \ \  \alpha \in [\underline{\alpha},1].$$
    Similarly, for every $b\geq 0, \alpha \in [\underline{\alpha},1]$ as $a$ is large enough, $\frac{\partial }{\partial a}J(t_i,0,\alpha, a,b)<0$. Thus proving that the supremum w.r.t $a$ can not be in $\{0,\infty\}$.
    \item Having fixed $k>1, b\geq 0, \alpha \in [\underline{\alpha},1]$, one considers the increasing function 
    $$\mathbb{R}_{+}\ni a\rightarrow\psi_{k,b,\alpha}(a)=-ak(q+q\frac{\mu}{\alpha}\theta(b,\alpha))+c(\alpha)\gamma(b,\alpha)+\lambda k \frac{\mathrm{e}^{-\frac{\mu}{\alpha} a}-1}{\frac{\mu}{\alpha}}.$$ 
    One easily gets $\left(\lim_{a\rightarrow 0+} \psi_{k,b,\alpha}(a)=c(\alpha)\gamma(b,\alpha)\right)\times \left(\lim_{a\rightarrow -\infty} \psi_{k,b,\alpha}(a)<0\right)$, and by the Mean Value Theorem, we conclude that the function $\psi_{k,b,\alpha}$ has at least one root in $(0,\infty)$, and since the function is monotone ($\frac{\partial }{\partial a}\psi_{k,b,\alpha}(a)<0$), so there exists a unique critical point $a_{k,b,\alpha}\ne 0$, satisfying $\frac{\partial }{\partial a}J(t_i,0,\alpha, a,b)=0$.
From \eqref{da0}, \ref{Jabf}, and \eqref{eq:mr}, we get the expression 
$$J(t_i,0,\alpha, a_{k,b,\alpha},b)=ka_{k,b,\alpha}.$$
\end{enumerate}
\end{proof}
\noindent In the following proposition, we draw some conclusions concerning the optimal parameters $a^*$, $b^*$, and $\alpha^*$. 
\begin{Proposition}
    If $(\alpha^*,a^*, b^*)$ realises the maximum of the quantity $J(t_i,0,\alpha, a,b)$ , then it holds
    \begin{enumerate}
        \item either $\alpha^{*}=1$ then we turn to the case that was solved by \cite{avram2021equity}.
        \item  Or $\alpha^{*}\in (\underline{\alpha},1)$, satisfies
        \begin{equation}\label{df0}
        \alpha^*\left(c'(\alpha^*) \gamma(b^*,\alpha^*)+c(\alpha^*) \gamma_{\alpha^*}(b^*,\alpha^*)\right)-\lambda  k m(a^*,\alpha^*)+ka^* \left(\frac{\mu }{\alpha^*} q \theta(b^*,\alpha^*)-\mu q \theta_{\alpha^*}(b^*,\alpha^*)\right)=0.
        \end{equation}
        and we have two cases
    \begin{enumerate}
    \item either $b^{*}= 0$, in which case (\ref{da0}) and (\ref{df0}) implies that $a^{*} , \ \alpha^{*}$ satisfies
    \begin{equation}\label{b0opta}
    -a^*kq+c(\alpha^{*})+\lambda k \frac{\mathrm{e}^{-\frac{\mu}{\alpha^{*}} a^{*}}-1}{\frac{\mu}{\alpha^{*}}}=0,
    \end{equation}
   \begin{equation}\label{b0optph}
    \alpha^{*}c'(\alpha^*)-\lambda  k m(a^{*},\alpha^{*})=0,
    \end{equation}
    \item or $b^{*} \in (0,\infty)$, satisfies
    \begin{equation}\label{bopt}
    c(\alpha^*)\gamma_{b^*}(b^*,\alpha^*)(q+\frac{\mu q}{\alpha^*} \theta(b^*,\alpha^*)+\lambda \mathrm{e}^{-\frac{\mu}{\alpha^*} a^*})-(c(\alpha^*)\gamma(b^*,\alpha^*)-\lambda km(a^*,\alpha^*))\frac{\mu q}{\alpha^*} \theta_{b^*}(b^*,\alpha^*)=0,
    \end{equation}
    \end{enumerate}
    ,
    \begin{equation}\label{aopt}
    a^{*}=\frac{c(\alpha^{*})\gamma_{b}(b^{*},\alpha^{*})}{k q \frac{\mu }{\alpha^{*}} \theta_{b}(b^{*},\alpha^{*})}\,,
    \end{equation}
    and
    \begin{equation}\label{jaopt}
        J(t_{i},0,\alpha^{*}, a^{*},b^{*} )=\frac{c(\alpha^{*})\gamma_{b}(b^{*},\alpha^{*})}{ q \frac{\mu }{\alpha^{*}} \theta_{b}(b^{*},\alpha^{*})}.
    \end{equation}
    In this case, the first eigenvalue of the Hessian matrix of $J(t_i, 0, \cdot, \cdot, \cdot)$ at a critical point $(\alpha, a, b)$ is negative. Hence all critical points are either inflexion or (local) maximum points.
    \end{enumerate}
\end{Proposition}
\begin{proof}
\begin{enumerate}
    \item This case is studied in \cite{avram2021equity}.
    \item The partial derivative of $J(t_i,0,\alpha, a,b)$ with respect to $\alpha$ (on $(\underline{\alpha}, 1)$) is given by
    \begin{align*}
    \frac{\partial }{\partial \alpha}J(t_i,0,\alpha, a,b)&=\frac{\left(c'(\alpha) \gamma(b,\alpha)+c(\alpha) \gamma_{\alpha}(b,\alpha)-\frac{\lambda  k}{\alpha} \left(m(a,\alpha)-\frac{\mu  a^{2}}{\alpha}\mathrm{e}^{-\frac{\mu  a}{\alpha}}\right)\right)\left(q+\frac{\mu  q \theta(b,\alpha)}{\alpha}+\lambda   \mathrm{e}^{-\frac{\mu  a}{\alpha}}\right)}{\left(q+\frac{\mu  q \theta(b,\alpha)}{\alpha}+\lambda   \mathrm{e}^{-\frac{\mu  a}{\alpha}}\right)^{2}},\\
    &-\frac{\left(c(\alpha) \gamma(b,\alpha)-\lambda  k m(a,\alpha)\right) \left(-\frac{\mu }{\alpha^{2}} q \theta(b,\alpha)+\frac{\mu }{\alpha} q \theta_{\alpha}(b,\alpha)+\lambda \frac{ \mu }{\alpha^{2}} a  \mathrm{e}^{-\frac{\mu  a}{\alpha}}\right)}{\left(q+\frac{\mu  q \theta(b,\alpha)}{\alpha}+\lambda   \mathrm{e}^{-\frac{\mu  a}{\alpha}}\right)^{2}}=0,
    \end{align*}
    with the first-order condition for $\alpha$ we obtain (\ref{df0}).
    \begin{enumerate}
        \item The assertions (\ref{b0opta}), and (\ref{b0optph}) in the first case $b^* = 0$ is a direct consequence of equation (\ref{df0}), and Lemma(\ref{akb}).
        \item In the case when $b^* \in (0, \infty)$, the $(\alpha^*,a^*, b^*)$ should be a critical point. The partial derivative of $J(t_i,0,\alpha, a,b)$ with respect to $b$ (on $(0, \infty)$) is given by
        \begin{equation*}
        \frac{\partial }{\partial b}J(t_i,0,\alpha, a,b)=\frac{c(\alpha)\gamma_{b}(b,\alpha)(q+\frac{\mu q}{\alpha} \theta(b,\alpha)+\lambda \mathrm{e}^{-\frac{\mu}{\alpha} a})-(c(\alpha)\gamma(b,\alpha)-\lambda km(a,\alpha))\frac{\mu q}{\alpha} \theta_{b}(b,\alpha)}{(q+\frac{\mu q}{\alpha} \theta(b,\alpha)+\lambda \mathrm{e}^{-\frac{\mu}{\alpha} a})^{2}}.
        \end{equation*}
     We write the first-order condition for b ($\frac{\partial }{\partial b}J(t_i,0,\alpha, a,b)=0$), for getting (\ref{bopt}).
    From (\ref{bopt}) and (\ref{da0}) we get (\ref{aopt}).
    Substituting (\ref{aopt}) into \ref{jofa}, we obtain (\ref{jaopt}).
Moreover, for the Hessian matrix at a critical point $(\alpha, a, b)$ we have
\begin{equation*}\label{d2/a2,b2}
    \frac{\partial^{2} }{\partial a^{2}}J(t_i,0,\alpha, a,b)<0, \qquad \quad \frac{\partial^{2} }{\partial b^{2}}J(t_i,0,\alpha, a,b)\leq 0,
\end{equation*}
\begin{equation*}\label{d2/ab}
    \frac{\partial^{2} }{\partial b \partial a}J(t_i,0,\alpha, a,b)=\frac{\partial^{2} }{\partial \alpha \partial a}J(t_i,0,\alpha, a,b)=\frac{\partial^{2} }{ \partial a \partial b}J(t_i,0,\alpha, a,b)=\frac{\partial^{2} }{ \partial a \partial \alpha}J(t_i,0,\alpha, a,b)=0\,,
\end{equation*}
\begin{align*}
    \frac{\partial^{2} }{ \partial \alpha \partial b}J(t_i,0,\alpha, a,b)&=\frac{\partial^{2} }{ \partial b \partial \alpha}J(t_i,0,\alpha, a,b)\\
    &=\frac{1}{q+\frac{\mu  q \theta(b,\alpha)}{\alpha}+\lambda   \mathrm{e}^{-\frac{\mu  a}{\alpha}}}\left[ c' \gamma_{b}+c(\alpha)\gamma_{\alpha b}+\frac{c(\alpha^{*})\gamma_{b}(b^{*},\alpha^{*})}{ \theta_{b}(b^{*},\alpha^{*})}\left( \frac{\theta_{b}}{\alpha}-\theta_{\alpha b}\right) \right]\,,\\
    \frac{\partial^{2} }{\partial \alpha^{2}}J(t_i,0,\alpha, a,b)&=\frac{1}{q+\frac{\mu  q \theta(b,\alpha)}{\alpha}+\lambda\mathrm{e}^{-\frac{\mu  a}{\alpha}}}\left[2c'\gamma_{\alpha}+c(\alpha)\gamma_{\alpha\alpha}\right.\\
    &\left.\qquad \quad -\frac{c(\alpha)\gamma_{b}}{\theta_{b}}\left(\frac{2}{\alpha^{2}}\theta-\frac{2}{\alpha}\theta_{\alpha}+\theta_{\alpha\alpha}\right)+\lambda  k \frac{\mu}{\alpha^{3}} \left(\frac{c(\alpha^{*})\gamma_{b}(b^{*},\alpha^{*})}{k q \frac{\mu }{\alpha^{*}} \theta_{b}(b^{*},\alpha^{*})}\right)^{2} \mathrm{e}^{-\frac{\mu}{\alpha} a} \right]\,
\end{align*}
and hence the statement follows.
 \end{enumerate}
\end{enumerate}
\end{proof}

\section{Numerical study}\label{sec:num}
In this section, we calculate the solution of the HJB equation (\ref{HJB1}) by using a policy iteration algorithm named Howard algorithm (\cite{howard1960dynamic}). We solve it for a fixed $a$ and we use finite difference approximations to discretise equation (\ref{HJB1}) as it is detailed below.
\subsection{Finite-difference scheme}
Let $N \in \mathbb{N}$ and $\delta>0$. We partition the space domain $D_\delta \subset (0,\infty)$ using a mesh $x_{\delta}=(x_{j})_{j=1,2,...,N}$, such that $x_{j}=j\delta$. The HJB (\ref{HJB1}) is discretised by replacing the first derivatives of $v$ with the following approximations
\begin{align}
v'(t, x_{j})&=\frac{v(t, x_{j}+\delta)-v(t, x_{j})}{\delta}, \quad c(\alpha)\geq 0\,,\nonumber\\
v'(t,x_{j})&=\frac{v(t, x_{j})-v(t, x_{j}-\delta)}{\delta}, \quad    \mbox{other ways}\,, \quad t\geq 0, \quad j=1,2,\ldots, N\,,\nonumber\\
v(t,0)&=ka\,, \nonumber\\  
v(t,x)&=kx+ka\,,   \quad \qquad \qquad \qquad   \mbox{if}\quad  x \in (-a, 0)\,, \nonumber \\ 
v(t,x)&=0\,,    \qquad  \qquad \qquad \qquad \qquad \mbox{if}\quad     x \in (-\infty, -a]\,.
\end{align}
This finite difference approximation leads to a system of $N$ equations with $N$ unknowns $(v_{\delta}(x_{j}))_{j=1,2,...,N}$ given as a solution to 
$$\max\{H^{\alpha}_{\delta}v_{\delta}, \ \varrho-Bv_{\delta}\}  =0,$$
where $H^{\alpha}_{\delta}$ is the matrix associated to the approximation of the operator $H(t_{i},x,v,v'(t_{i},x))$, $B=(B^{ij})_{1\leq i,j\leq N}$ is the matrix associated to the second term of our equation, defined as 
$$\left\{ \begin{array}{cl}
B^{ii}=1, \\
B^{(i-1)i}=-1,  \\
B^{ji}=0\,, &\mbox{for all } j\notin \{i , i-1\}\,,
\end{array} \right.$$
and $\varrho= (\varrho^i)_{1\leq i\leq N}$ is the vector which all entries equal $1$ except $\varrho^1=1+ak/\delta$. In our numerical example, for simplicity we fix $a$ and we derive the values of $\alpha^*$ and $b^*$.\\
{\it The Howard algorithm.}
To solve (~\ref{HJB1}), we use the Howard  algorithm, also named policy iteration. It consists of computing two sequences $\{\alpha^{n},D_{\delta,1}^{n},D_{\delta,2}^{n}\}_{n\geq 1}$ and $(v^{n}_{\delta})_{n\geq 1}$, (starting from $v^{1}_{\delta}$) defined by
\begin{enumerate}
    \item {\it Step 2n-1.} Given $v^{n}_{\delta}$, compute a feedback strategy $\alpha^{n} $ defined as
    $$\alpha^{n} =\mbox{argmax}_{\alpha \in [\underline{\alpha}, 1] }\{H^{\alpha}_{\delta}v^{n}_{\delta} \}\,.$$
    Define a new partition ($D_{\delta,1}^{n+1},D_{\delta,2}^{n+1}$) of $D_{\delta}$ as follows: compare the values of $H^{\alpha}_{\delta}v^{n}_{\delta}$ and $\varrho-Bv^{n}_{\delta} $ at all points of the grid. Define $D_{\delta,1}^{n+1}$ as the set of points of $D_{\delta}$ where  $H^{\alpha}_{\delta}v^{n}_{\delta}\geq b-Bv^{n}_{\delta}$ and $D_{\delta,2}^{n+1}= D_{\delta}\setminus D_{\delta,1}^{n+1}$.
    \item {\it Step 2n.} Define $v^{n+1}_{\delta}$ has the solution of the linear systems:
    $$H^{\alpha^{n}}_{\delta}v^{n+1}_{\delta}=0 \mbox{ in } D_{\delta,1}^{n},$$ 
    and
    $$\varrho-Bv^{n+1}_{\delta}=0 \mbox{ in }  D_{\delta,2}^{n}.$$
    \item If $|v^{n+1}_{\delta}-v^{n}_{\delta}|,<\varepsilon$ stop, otherwise, go to step 2n + 1.
\end{enumerate}
The sequence $(v^{n}_{\delta})_{n\geq 1}$ converges to the solution $v_{\delta}$ of ~\ref{HJB1} and the sequence of partition $\{D_{\delta,1}^{n},D_{\delta,2}^{n}\}_{n\geq 1}$ converges towards the optimal partition $\{D_{\delta,1},D_{\delta,2}\}$ of $D_{\delta}$, since the matrices $H^{\alpha^{n}}_{\delta}$
are diagonally dominant (see \cite{chancelier2007policy}).

\subsection{Numerical results}
The HJB equation (\ref{HJB1}) has been solved by using the Howard algorithm. This algorithm is very efficient and converges fastly to the solution.
The optimal policy has the following form: When the surplus process is in $(-a; 0)$, the insurer injects the reserve. When the surplus process is in $(-a; b)$, the insurer does not distribute any dividend, and only pay dividends when the reserve above some critical threshold $b$, as mentioned in the following figures.
In the following, we use the model parameters given by $U_{i}\sim U_{[b_{min}, b_{max}}]$  distributed claim amounts. The insurer’s premium rate is determined via the expected value principle and reads $ c=(1+\eta_{1})\lambda\mathbb{E}[U]=(1+\eta_{1})\lambda\mu.$
In the following table, we add the value of the parameters that we use for our model.
\begin{table}[h]
    \centering
    \begin{tabular}{|c|c|c|c|c|c|c|c|c|c|c|c|} 
    \hline
    The parameter & N & $\delta$ & $\eta_1$ & $\eta_2$ & $q$ & $k$ & $a$ & $\lambda$& $b_{min}$& $b_{max}$&$\mu$\\ \hline
    The value & 300 & 0.009 & 0.1 & 0.11 & 0.15 & 1.14 & 0.85 & 4 &0.1&0.9&0.5 \\ \hline
    \end{tabular}
    \label{tab1}
\end{table}

\noindent Figure \ref{ORS} illustrates a positive correlation between the optimal reinsurance strategy $\alpha^{*}$ and the capital. This means that, by increasing the company's capital, its ability to absorb risks increases.
\begin{figure}[H]
\centering
\includegraphics[width=0.9\linewidth]{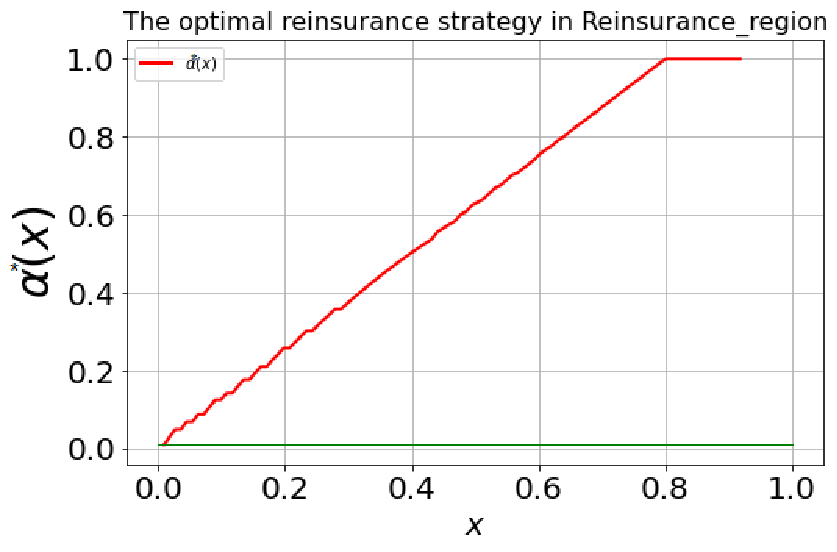}
\caption{The optimal reinsurance strategy $\alpha^{*}$}
	\label{ORS}
\end{figure}
\noindent We notice in Figure ~\ref{OVF} that the optimal value function increases with increasing capital and that there is a corner point when $x=0$, (see Figure ~\ref{OVF(OV)}). This explains that the value function is not differentiable at $x=0$.
\begin{figure}[H]
\centering
\includegraphics[width=0.9\linewidth]{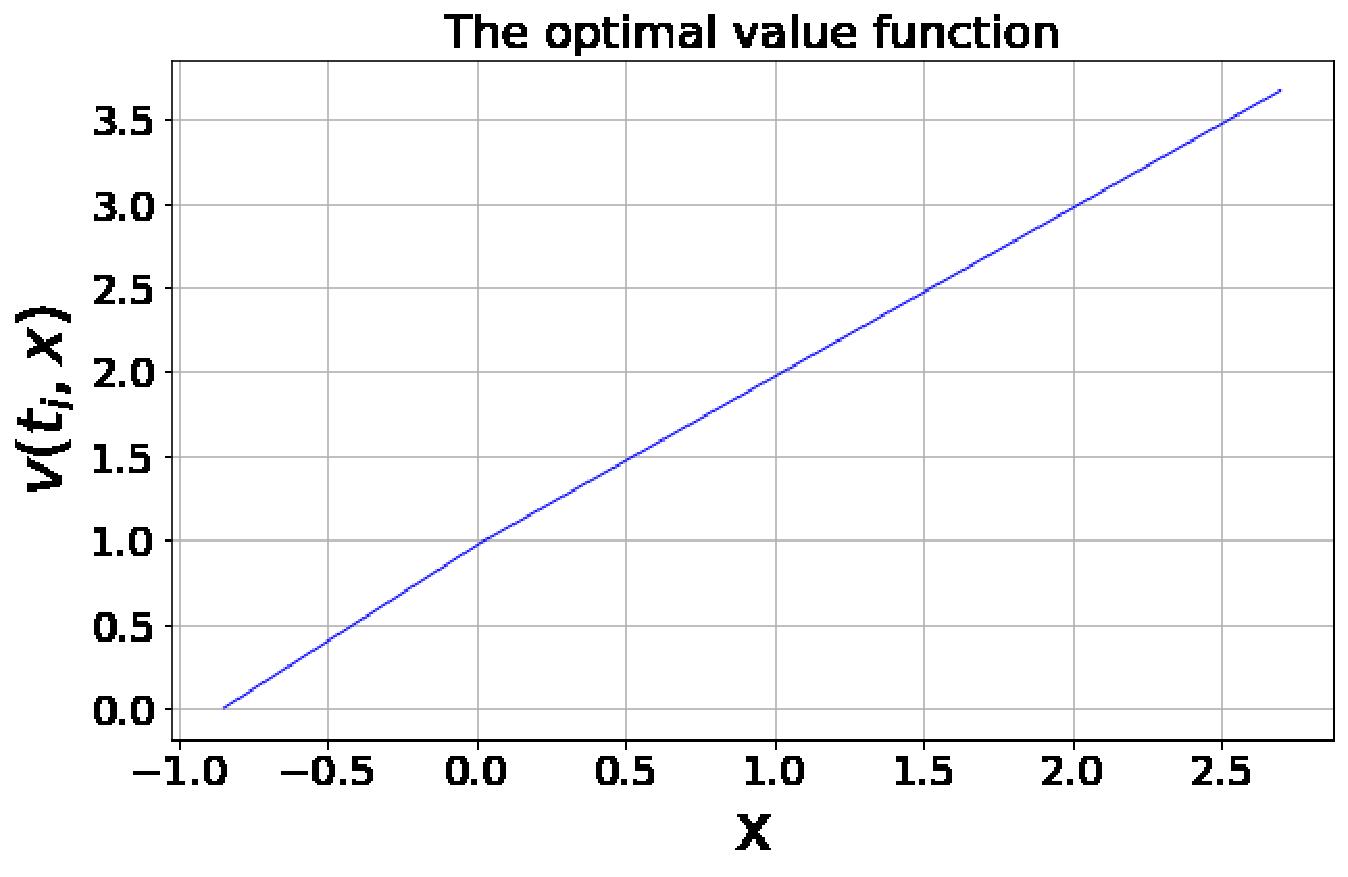}
\caption{The optimal value function}
	\label{OVF}
\end{figure}
\begin{figure}[H]
\centering
\includegraphics[width=0.9\linewidth]{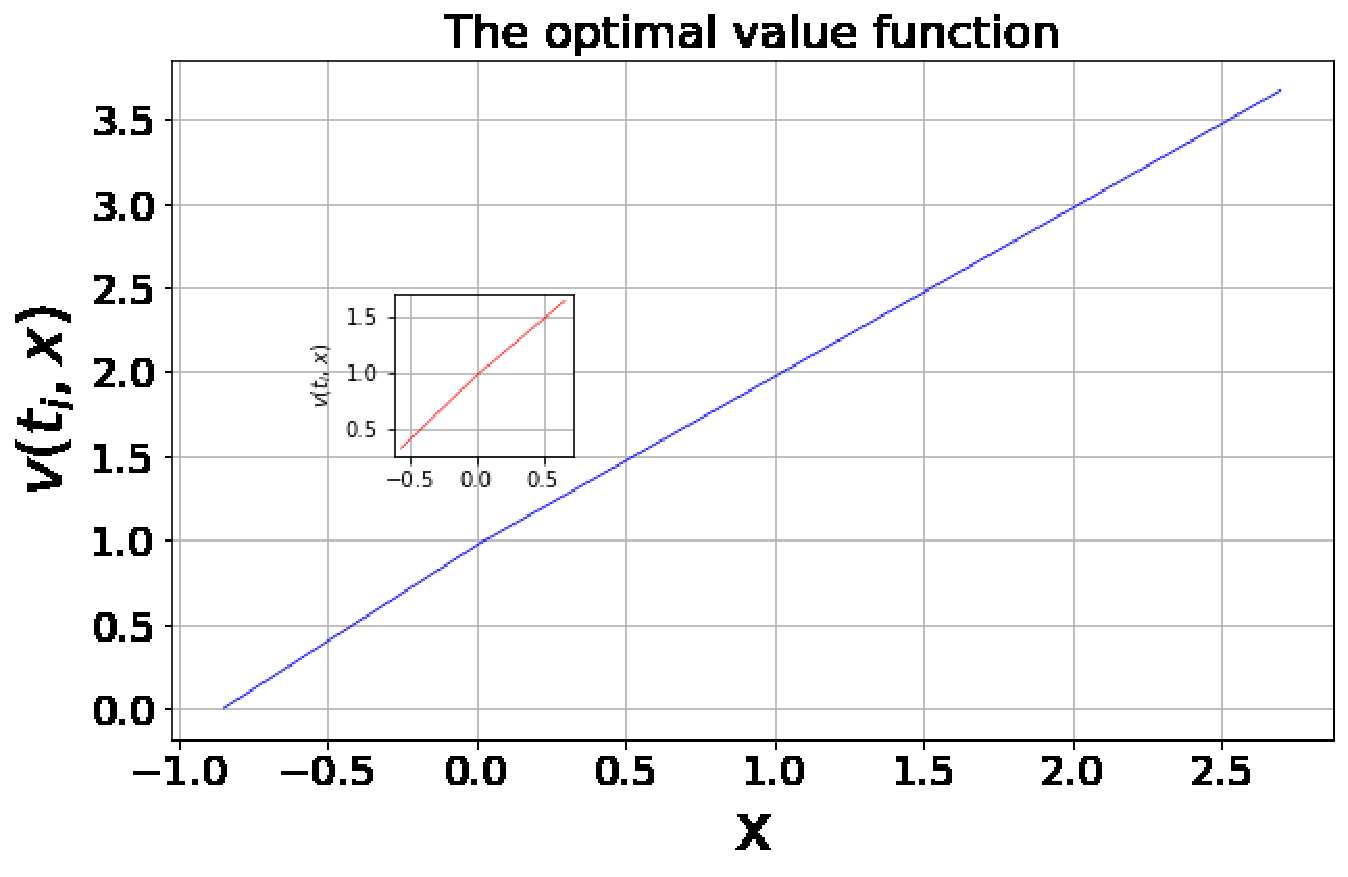}
\caption{The optimal value function (value function at x=0)}
	\label{OVF(OV)}
\end{figure}
\noindent Figure~\ref{OVFRS} illustrates both the optimal value function and the optimal reinsurance strategy $\alpha^{*}$. It is divided into three regions: capital injection, reinsurance, and dividends. The blue line represents the optimal value function, the red line represents the optimal reinsurance strategy $\alpha^{*}$, which begins from $\underline{\alpha}>0$, increases linearly in the reinsurance regions until the reserve reaches a critical threshold $b^*=0.8$. 
\begin{figure}[H]
\centering
\includegraphics[width=0.8\linewidth]{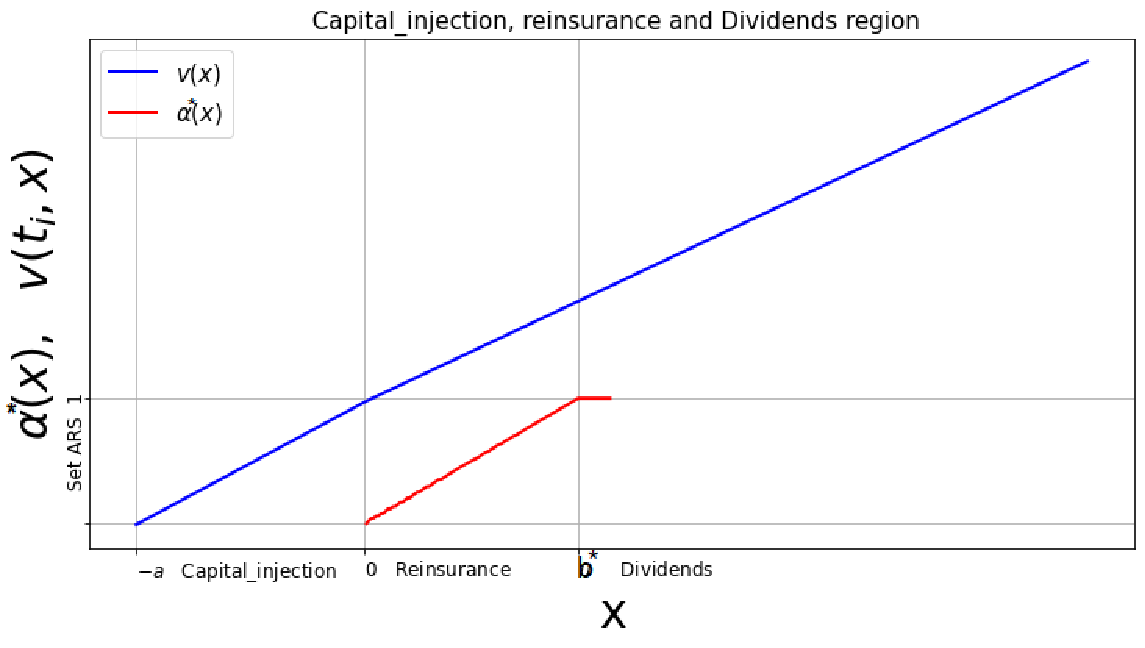}
\caption{The optimal value function and reinsurance}
	\label{OVFRS}
\end{figure}

\noindent Figure~\ref{SPWR} shows the effect of the reinsurance process on the insurance company's surplus process, as we note that the slope and the size of claims when there is no reinsurance $\alpha=1$, are greater than the case where the reinsurance is given by $\alpha^*(X_{.}^{t,x,\alpha,\pi})$.
\begin{figure}[H]
\centering

\includegraphics[width=0.7\linewidth]{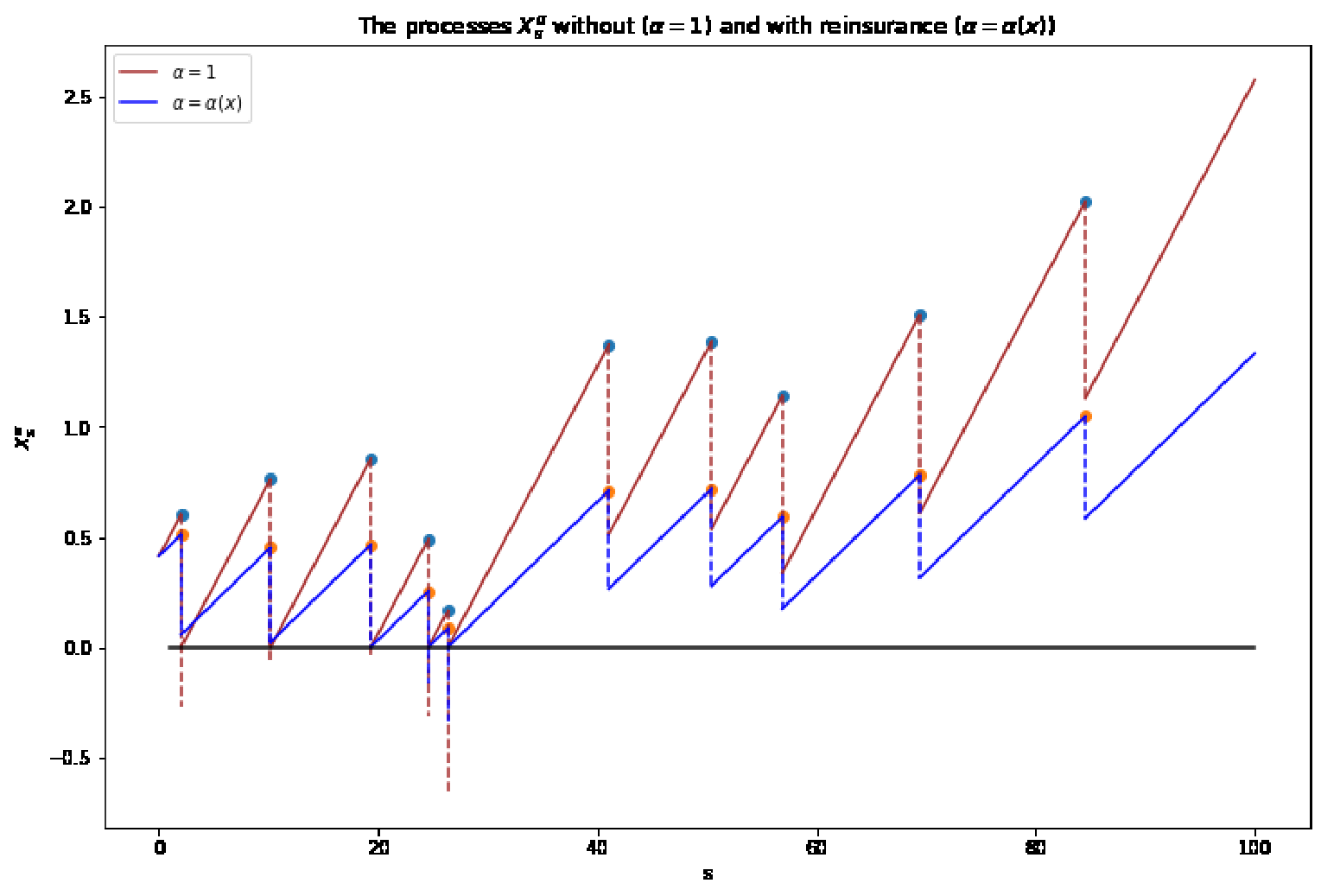}
\caption{surplus process with and without reinsurance strategy}
	\label{SPWR}
\end{figure}
\noindent We recall that we consider a fixed reinsurance rate, that is, on each interval $[t_i, t_{i+1})$, we have $\alpha^*t = \alpha{t_i}$.
Figures~\ref{SPWR112} and~\ref{SPWR136} illustrate the impact of the frequency of changes in the reinsurance strategy on the insurance company's surplus process. We observe that the surplus increases more rapidly as the reinsurance strategy is updated more frequently. Specifically, Figure~\ref{SPWR112} shows the evolution of the surplus when the company adjusts its reinsurance strategy every 1 and 12 time units, while Figure~\ref{SPWR136} depicts the corresponding results for updates every 1, 3, and 6 time units.

\begin{figure}[H]
\centering

\includegraphics[width=0.9\linewidth]{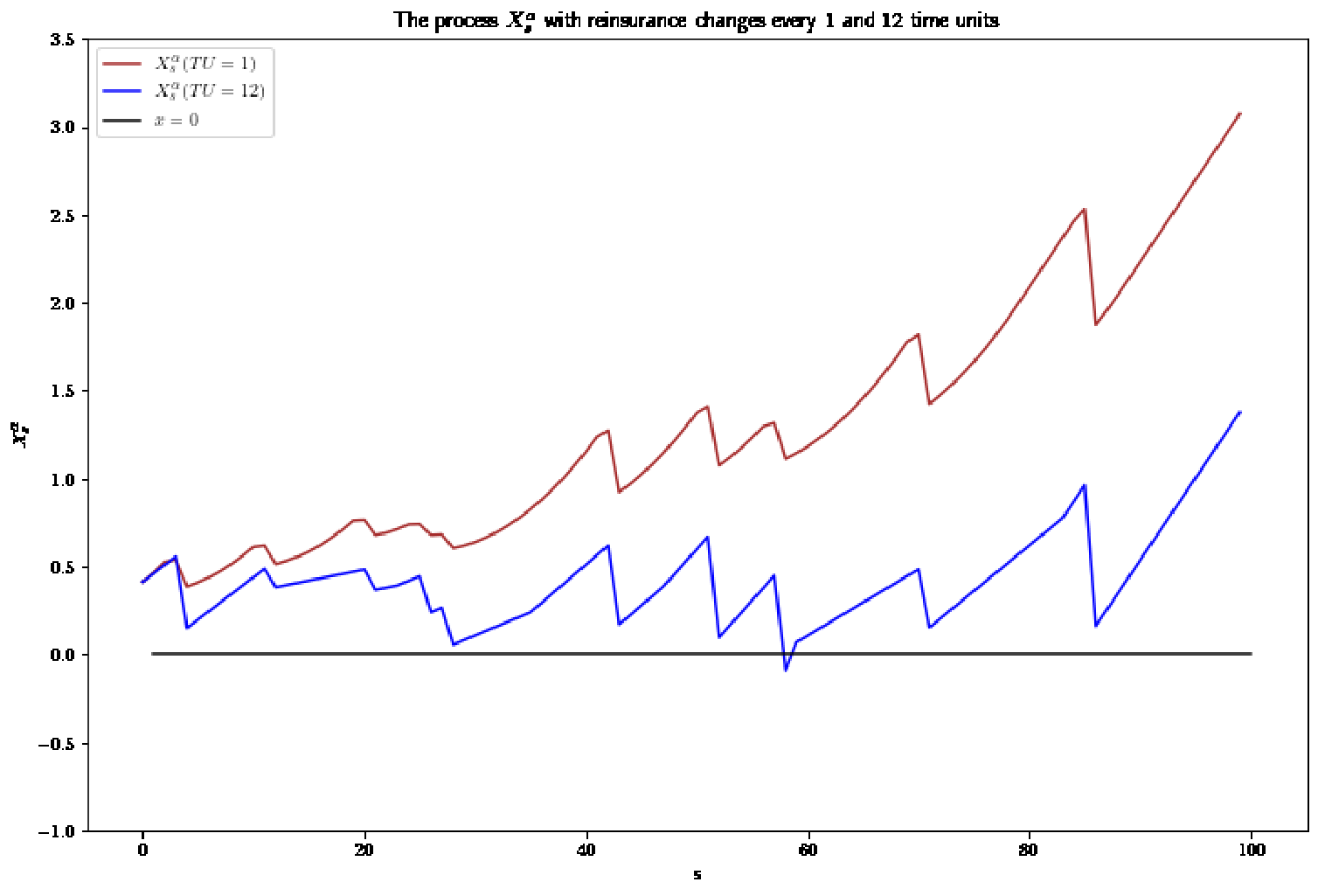}

\caption{The surplus process with Reinsurance (changes every 1 and 12 time units)}
	\label{SPWR112}
\end{figure}

\begin{figure}[H]
\centering

\includegraphics[width=0.9\linewidth]{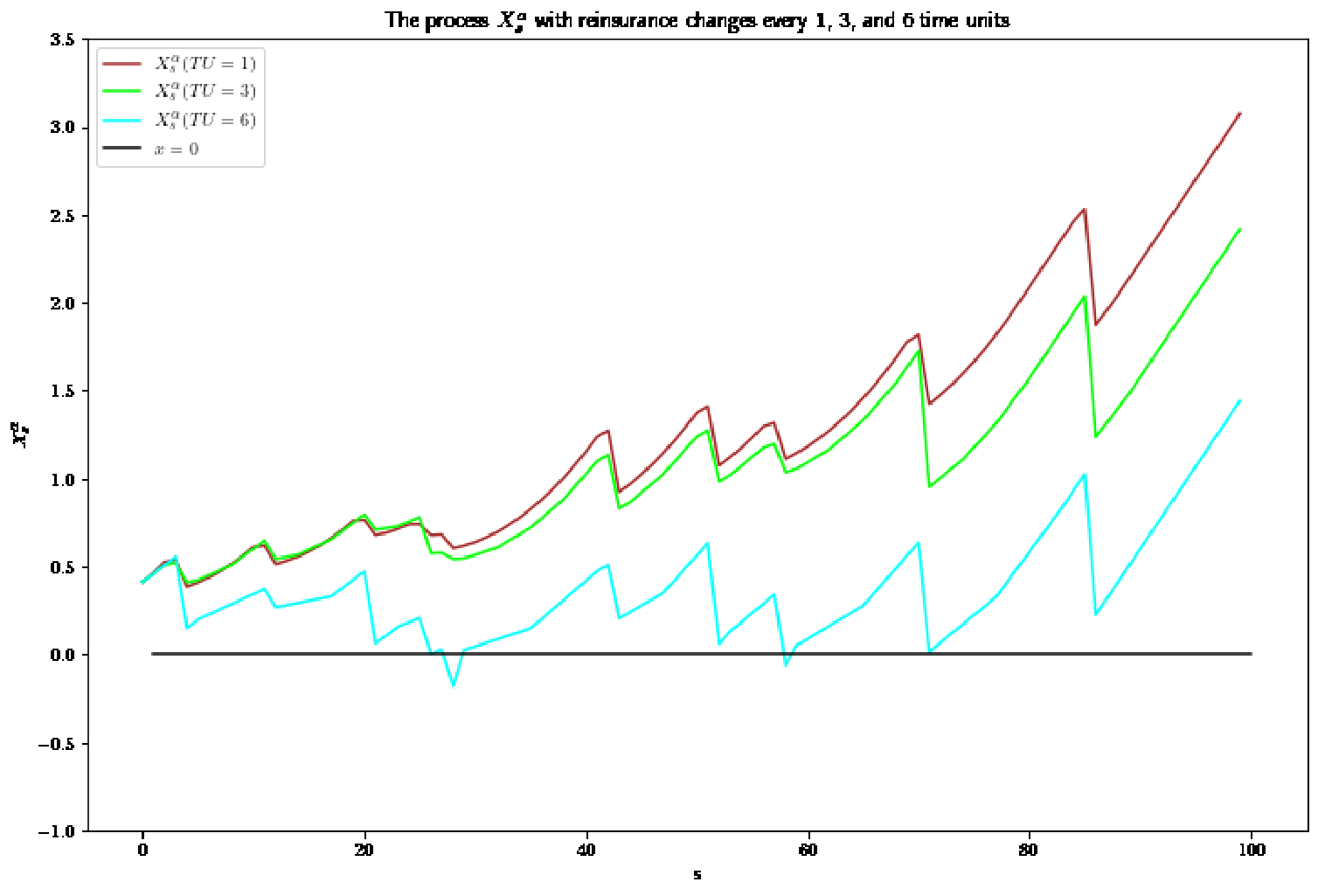}

\caption{The surplus process with Reinsurance (changes every 1, 3 and 6 time units)}
	\label{SPWR136}
\end{figure}
\noindent Figure~\ref{ORST} illustrates the optimal reinsurance strategy $\alpha^{*}_{t}$ over a span of 100 time units $[t_{i}, t_{i+1}),\ \ i=0,1,2,...,99$, presented as a function of time. In this context, the insurance company decides the level of retention at the onset of each time unit by specifying the $\alpha^*_{t_{i}}$ value at the beginning of the period $[t_{i}, t_{i+1})$, and $\alpha^*_{t}=\alpha^*_{t_i}$ during this period. We recall that the red line represents the lowest admissible retention.

\begin{figure}[H]
\centering

\includegraphics[width=0.8\linewidth]{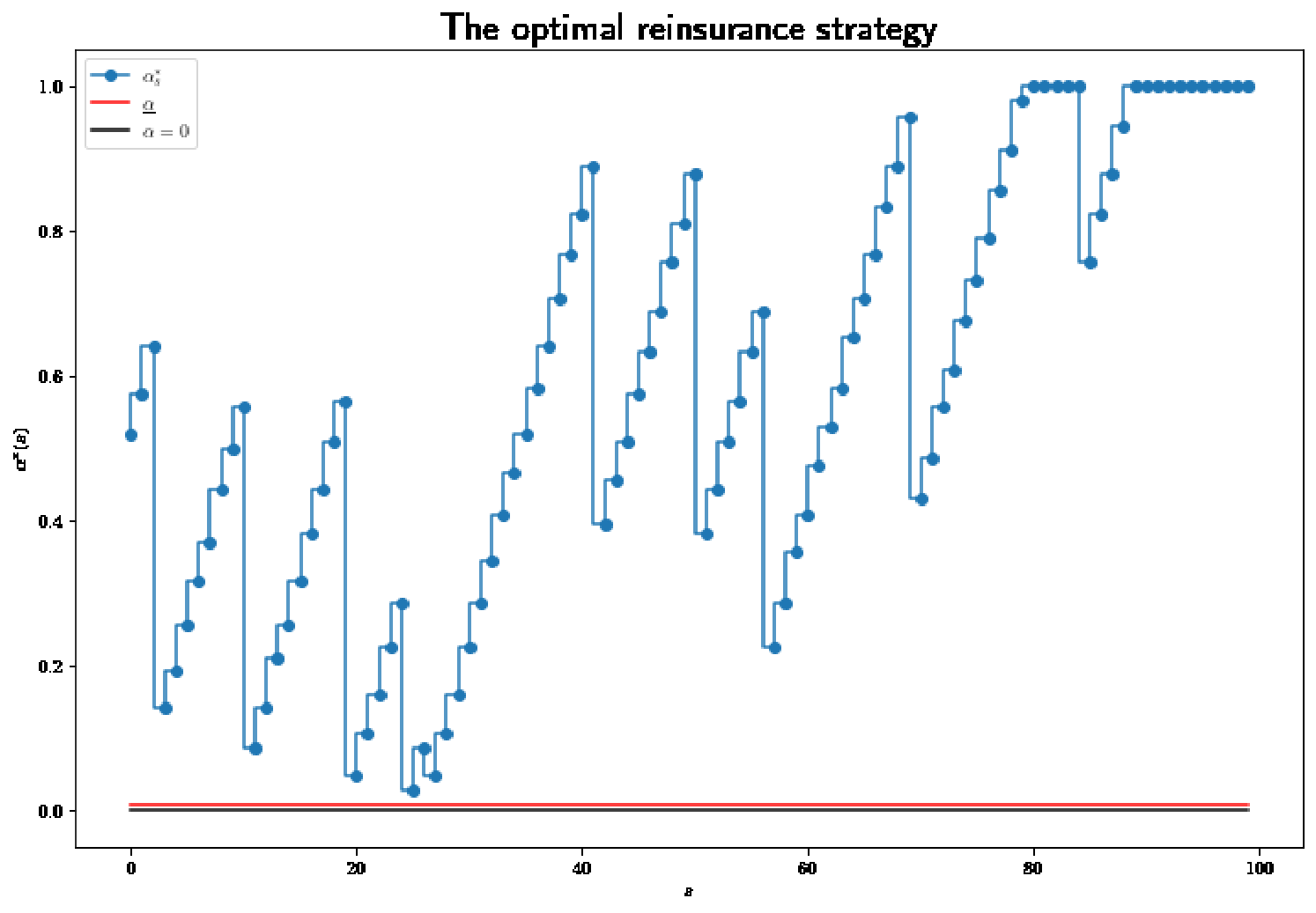}
\caption{The optimal reinsurance strategy $\alpha^{*}$ (changes every time unit)}
	\label{ORST}
\end{figure}

\begin{figure}[H]
\centering

\includegraphics[width=0.8\linewidth]{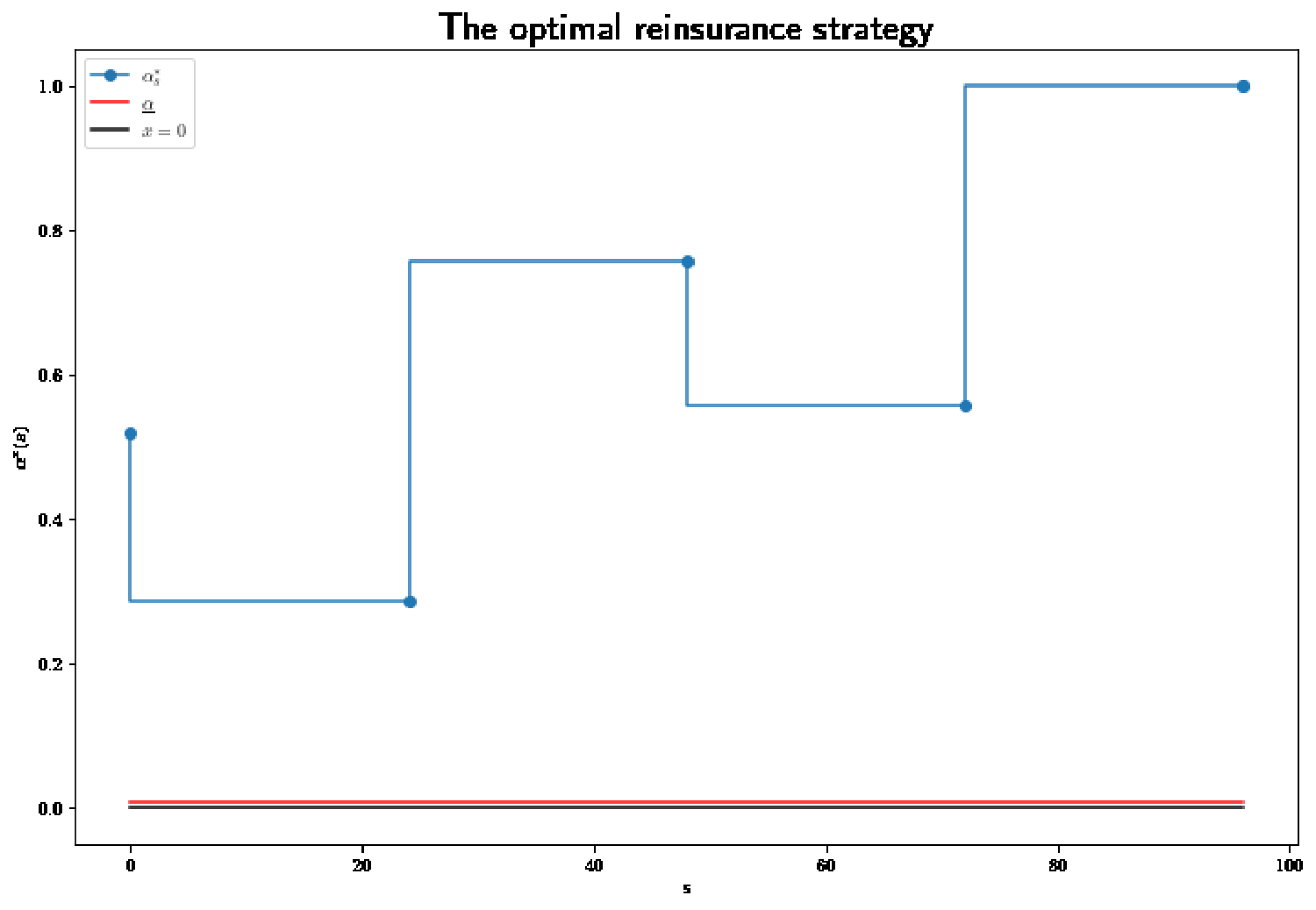}
\caption{The optimal reinsurance strategy $\alpha^{*}$ (changes every 12 time units)}
	\label{ORST12}
\end{figure}

\section{Conclusion} \label{sec:conc}
In this paper, we studied a dynamic optimal reinsurance problem which includes a stochastic process representing the reinsurance in the Cram\'er-Lundberg model. The aim was to determine the optimal limit for reducing the risks that the insurance company can bear. 
To solve our problem we used two approaches.
In the first approach, we proved that the value function maximising the difference between the expected discounted dividends and the injection of capitals of the insurance company is a solution to an HJB system in the viscosity sense.
We proved the uniqueness of the value function on $(-a,0)\cup(0,\infty)$. 
In the second approach, we used the L\'evy property of the state process in the interval $[t_i, t_{i+1})$ for all integer $i$. We get interesting relations between $\alpha^*$, $a^*$, and $b^*$. 
Finally, in our a numerical analysis, we used finite-difference method to solve the HJB equation by fixing first $a$ and solving for the optimal $\alpha$ and $b$.
Notice that the first approach allows for more general Markovian models to describe the claims and we do not have to restrict to L\'evy type as in the second approach.

\section*{Acknowledgement}
The authors thank the anonymous reviewer for their valuable comments and constructive suggestions, which have improved the quality of this paper.

\bibliographystyle{plain}
\bibliography{refrance.bib}

@article{di2024utility,
  title={Utility maximisation and change of variable formulas for time-changed dynamics},
  author={Di Nunno, G. and Haferkorn, H. and Khedher, A. and Vanmaele, M.},
  journal={arXiv preprint arXiv:2407.02915},
  year={2024}
}

@book{mandjescramer,
  title={The {C}ram{\'e}r-{L}undberg Model and Its Variants. {A} queuing perspective},
  author={Mandjes, M. and Boxma, O.},
  year={2003},
  publisher={Springer}
}

@article{avram2015gerber,
  title={On {G}erber--{S}hiu functions and optimal dividend distribution for a L{\'e}vy risk process in the presence of a penalty function},
  author={Avram, F. and Palmowski, Z. and Pistorius, M. R.},
journal={Annales of Applied Probability},
  year={2015},
  volume={25},
 number={4},
  pages={1868--1935},
}

@article{cohen2020rate,
  title={Rate of convergence of the probability of ruin in the {C}ram{\'e}r--{L}undberg model to its diffusion approximation},
  author={Cohen, A. and Young, V. R.},
  journal={Insurance: Mathematics and Economics},
  volume={93},
  pages={333--340},
  year={2020},
  publisher={Elsevier}
}

@article{irgens2004optimal,
  title={Optimal control of risk exposure, reinsurance and investments for insurance portfolios},
  author={Irgens, C. and Paulsen, J.},
  journal={Insurance: Mathematics and Economics},
  volume={35},
  number={1},
  pages={21--51},
  year={2004},
  publisher={Elsevier}
}

@article{schal2004discrete,
  title={On discrete-time dynamic programming in insurance: exponential utility and minimizing the ruin probability},
  author={Sch{\"a}l, M.},
  journal={Scandinavian Actuarial Journal},
  volume={2004},
  number={3},
  pages={189--210},
  year={2004},
  publisher={Taylor \& Francis}
}

@article{eisenberg2011optimal,
  title={Optimal control of capital injections by reinsurance with a constant rate of interest},
  author={Eisenberg, J. and Schmidli, H.},
  journal={Journal of applied probability},
  volume={48},
  number={3},
  pages={733--748},
  year={2011},
  publisher={Cambridge University Press}
}

@article{kulenko2008optimal,
  title={Optimal dividend strategies in a {C}ram{\'e}r--{L}undberg model with capital injections},
  author={Kulenko, N. and Schmidli, H.},
  journal={Insurance: Mathematics and Economics},
  volume={43},
  number={2},
  pages={270--278},
  year={2008},
  publisher={Elsevier}
}

@article{avanzi2011optimal,
  title={Optimal dividends and capital injections in the dual model with diffusion},
  author={Avanzi, B. and Shen, J. and Wong, B.},
  journal={ASTIN Bulletin: The Journal of the IAA},
  volume={41},
  number={2},
  pages={611--644},
  year={2011},
  publisher={Cambridge University Press}
}

@article{avram2022optimizing,
  title={Optimizing dividends and capital injections limited by bankruptcy, and practical approximations for the {C}ram{\'e}r-{L}undberg process},
  author={Avram, F. and Goreac, D. and Adenane, R. and Solon, U.},
  journal={Methodology and Computing in Applied Probability},
  volume={24},
  number={4},
  pages={2339--2371},
  year={2022},
  publisher={Springer}
}

@article{touzi2000optimal,
  title={Optimal insurance demand under marked point processes shocks},
  author={Touzi, N.},
  journal={Annals of Applied Probability},
  volume={10},
  number={1},
  pages={283--312},
  year={2000},
  publisher={JSTOR}
}

@article{li2015optimal,
  title={Optimal Dividend and Capital Injection Strategies in the {C}ram{\'e}r-{L}undberg Risk Model},
  author={Li, Y. and Liu, G.},
  journal={Mathematical Problems in Engineering},
  volume={2015},
  pages={},
  year={2015},
  publisher={Wiley Online Library}
}

@book{jacod2013limit,
  title={Limit theorems for stochastic processes},
  author={Jacod, J. and Shiryaev, A.},
  volume={288},
  year={2013},
  publisher={Springer Science \& Business Media}
}

@book{revuz2013continuous,
  title={Continuous martingales and {B}rownian motion},
  author={Revuz, D. and Yor, M.},
  EDITION = {{F}irst},
  year={1991},
  publisher={Springer-Verlag, Berlin, Heidelberg}
}

@book{protter2004stochastic,
  title={Stochastic Integration and Differential equations},
  author={Protter, P.},
  EDITION = {{S}econd},
  year={2005},
  PUBLISHER = {Springer, Berlin},
}

@article{cani2017optimal,
  title={An optimal reinsurance problem in the {C}ram{\'e}r--{L}undberg model},
  author={Cani, A. and Thonhauser, S.},
  journal={Mathematical methods of operations research},
  volume={85},
  number={2},
  pages={179--205},
  year={2017},
  publisher={Springer}
}

@article{schmidli2001optimal,
  title={Optimal proportional reinsurance policies in a dynamic setting},
  author={Schmidli, H.},
  journal={Scandinavian Actuarial Journal},
  volume={2001},
  number={1},
  pages={55--68},
  year={2001},
  publisher={Taylor \& Francis}
}

@article{hipp2003optimal,
  title={Optimal dynamic {X}{L} reinsurance},
  author={Hipp, C. and Vogt, M.},
  journal={ASTIN Bulletin: The Journal of the IAA},
  volume={33},
  number={2},
  pages={193--207},
  year={2003},
  publisher={Cambridge University Press}
}

@book{schmidli2007stochastic,
  title={Stochastic control in insurance},
  author={Schmidli, H.},
  year={2007},
  publisher={Springer Science \& Business Media}
}

@article{hipp2010optimal,
  title={Optimal non-proportional reinsurance control},
  author={Hipp, C. and Taksar, M.},
  journal={Insurance: Mathematics and Economics},
  volume={47},
  number={2},
  pages={246--254},
  year={2010},
  publisher={Elsevier}
}

@article{eisenberg2010optimal,
  title={On optimal control of capital injections by reinsurance and investments},
  author={Eisenberg, J.},
  journal={Bl{\"a}tter der DGVFM},
  volume={31},
  number={2},
  pages={329--345},
  year={2010},
  publisher={Springer}
}

@article{azcue2005optimal,
  title={Optimal reinsurance and dividend distribution policies in the Cram{\'e}r-Lundberg model},
  author={Azcue, P. and Muler, N.},
  journal={Mathematical Finance: An International Journal of Mathematics, Statistics and Financial Economics},
  volume={15},
  number={2},
  pages={261--308},
  year={2005},
  publisher={Wiley Online Library}
}

@article{mnif2005optimal,
  title={Optimal risk control and dividend policies under excess of loss reinsurance},
  author={Mnif, M. and Sulem, A.},
  journal={Stochastics An International Journal of Probability and Stochastic Processes},
  volume={77},
  number={5},
  pages={455--476},
  year={2005},
  publisher={Taylor \& Francis}
}

@article{jgaard1999controlling,
  title={Controlling risk exposure and dividends payout schemes: insurance company example},
  author={Jgaard, B. and Taksar, M.},
  journal={Mathematical Finance},
  volume={9},
  number={2},
  pages={153--182},
  year={1999},
  publisher={Wiley Online Library}
}

@article{schal1998piecewise,
  title={On piecewise deterministic Markov control processes: control of jumps and of risk processes in insurance},
  author={Sch{\"a}l, M.},
  journal={Insurance: Mathematics and Economics},
  volume={22},
  number={1},
  pages={75--91},
  year={1998},
  publisher={Elsevier}
}

@book{azcue2014stochastic,
  title={Stochastic optimization in insurance: a dynamic programming approach},
  author={Azcue, P. and Muler, N.},
  year={2014},
  publisher={Springer}
}

@book{asmussen2010ruin,
  title={Ruin probabilities},
  author={Asmussen, S. and Albrecher, H.},
  volume={14},
  year={2010},
  publisher={World scientific}
}

@Article{avram2021equity,
AUTHOR = {Avram, F. and Goreac, D. and Li, J. and Wu, X.},
TITLE = {Equity Cost Induced Dichotomy for Optimal Dividends with Capital Injections in the {C}ram{\'e}r-{L}undberg Model},
JOURNAL = {Mathematics},
VOLUME = {9},
YEAR = {2021},
NUMBER = {9},
ARTICLE-NUMBER = {931},
URL = {https://www.mdpi.com/2227-7390/9/9/931},
ISSN = {2227-7390},
}

@article{de1957impostazione,
  title={Su un’impostazione alternativa della teoria collettiva del rischio},
  author={De Finetti, B.},
  journal={Transactions of the XVth International Congress of Actuaries, New York},
  volume={2},
  number={1},
  pages={433--443},
  year={1957},
  organization={New York}
}

@article{shreve1984optimal,
  title={Optimal consumption for general diffusions with absorbing and reflecting barriers},
  author={Shreve, S. and Lehoczky, J. and Gaver, D.},
  journal={SIAM Journal on Control and Optimization},
  volume={22},
  number={1},
  pages={55--75},
  year={1984},
  publisher={SIAM}
}

@article{junca2019optimal,
  title={Optimal bail-out dividend problem with transaction cost and capital injection constraint},
  author={Junca, M. and Moreno-Franco, H. and P{\'e}rez, J.},
  journal={Risks},
  volume={7},
  number={1},
  pages={13},
  year={2019},
  publisher={Multidisciplinary Digital Publishing Institute}
}

@book{kyprianou2013gerber,
  title={Gerber--Shiu risk theory},
  author={Kyprianou, A. E.},
  year={2013},
  publisher={Springer Science \& Business Media}
}

@book{pham2009continuous,
  title={Continuous-time stochastic control and optimization with financial applications},
  author={Pham, H.},
  year={2009},
  publisher={Springer Science \& Business Media}
}

@article{chancelier2007policy,
  title={A policy iteration algorithm for fixed point problems with nonexpansive operators},
  author={Chancelier, J. and Messaoud, M. and Sulem, A.},
  journal={Mathematical Methods of Operations Research},
  volume={65},
  pages={239--259},
  year={2007},
  publisher={Springer}
}

@book{howard1960dynamic,
  title={Dynamic programming and markov processes.},
  author={Howard, R. A.},
  year={1960},
  publisher={John Wiley}
}

@book{kyprianou2014fluctuations,
  title={Fluctuations of L{\'e}vy processes with applications: Introductory Lectures},
  author={Kyprianou, A. E.},
  year={2014},
  publisher={Springer Science \& Business Media}
}

@misc{renaud2023optimizationdichotomycapitalinjections,
      title={An optimization dichotomy for capital injections and absolutely continuous dividend strategies}, 
      author={J.F. Renaud and A. Roch and Clarence S.},
      year={2023},
      eprint={2311.10191},
      archivePrefix={arXiv},
      primaryClass={math.OC},
      url={https://arxiv.org/abs/2311.10191}, 
}
\end{document}